\documentclass[a4paper]{amsart}
\usepackage[new]{old-arrows}
\usepackage[T1]{fontenc}
\usepackage[english]{babel}
\usepackage{kpfonts}

\usepackage{amsmath,amssymb,amsthm,enumerate,xspace}
\usepackage{comment}
\usepackage{accents}
\usepackage[colorlinks=true,citecolor=blue,linkcolor=blue]{hyperref}
\usepackage[all]{xy}
\usepackage{tikz}
\usetikzlibrary{arrows}
\usepackage{url}

\usepackage{mathrsfs}
\usepackage{cleveref}
\usepackage{graphicx} 

\numberwithin{equation}{section}

\newtheorem{thm}{Theorem}[section]

\newtheorem{prop}[thm]{Proposition}
\newtheorem{lem}[thm]{Lemma}

\newtheorem*{openproblem*}{Problem}
\newtheorem*{quest*}{Question}
\newtheorem*{problem*}{Problem}
\newtheorem*{claim*}{Claim}

\theoremstyle{definition}
\newtheorem{defn}[thm]{Definition}

\theoremstyle{remark}
\newtheorem{rem}[thm]{Remark}

\newcommand{\bD}{\mathbb{D}}

\newcommand{\bR}{\mathbb{R}}
\newcommand{\bS}{\mathbb{S}}

\newcommand{\bZ}{\mathbb{Z}}

\newcommand\Diff{\mathrm{Diff}}
\newcommand\Homeo{\mathrm{Homeo}}

\newcommand\BDiff{\mathrm{BDiff}}

\newcommand{\hcoker}{/\!\!/}

\newcommand{\tH}{\text{\textnormal{Homeo}}}

\newcommand{\BD}{\mathrm{B}\text{\textnormal{Homeo}}}

\newcommand{\BdD}{\mathrm{B}\text{\textnormal{Homeo}}^{\delta}}
\newcommand{\dD}{\text{\textnormal{Homeo}}^{\delta}}

\newcommand{\rotcong}{\rotatebox[origin=c]{90}{$\cong$}}
\makeatother
\addtocontents{toc}{\protect\setcounter{tocdepth}{1}}
\makeatletter
\let\c@equation\c@thm
\makeatother
\numberwithin{equation}{section}
\title[]{On the finiteness of the classifying space of diffeomorphisms of reducible three manifolds}

\author{Sam Nariman}
\address{Department of Mathematics\\
  Purdue University\\
150 N. University Street\\
West Lafayette, IN 47907-2067\\
}
\email{snariman@purdue.edu}

\begin{document}
\begin{abstract}
 Kontsevich (\cite[Problem 3.48]{kirby1995problems}) conjectured that $\BDiff(M, \text{rel }\partial)$ has the homotopy type of a finite CW complex for all compact $3$-manifolds with non-empty boundary. Hatcher-McCullough (\cite{MR1486644}) proved this conjecture when $M$ is irreducible. We prove a homological version of Kontsevich's conjecture. More precisely, we show that  $\BDiff(M, \text{rel }\partial)$ has finitely many nonzero homology groups each finitely generated when $M$ is a connected sum of irreducible $3$-manifolds that each have a nontrivial and non-spherical boundary.
\end{abstract}
\maketitle
\section{Introduction}

For a closed surface $\Sigma_g$ of genus $g>1$, it is well-known that the classifying space $\BDiff(\Sigma_g)$ is rationally equivalent to $\mathcal{M}_g$, the moduli space of Riemann surfaces of genus $g$. Therefore, in particular, the rational homology groups of $\BDiff(\Sigma_g)$ vanish above a certain degree, and in fact, more precisely they vanish above degree $4g-5$, which is the virtual cohomological dimension of the mapping class group $\text{Mod}(\Sigma_g)$. For a surface $\Sigma_{g,k}$ with $k>0$ boundary components, the classifying space $\BDiff(\Sigma_{g,k}, \text{rel }\partial)$ is in fact homotopy equivalent to the corresponding moduli space of Riemann surfaces of genus $g$ with $k$ boundary components. Therefore, $\BDiff(\Sigma_{g,k}, \text{rel }\partial)$ has the homotopy type of a finite-dimensional CW-complex. 

Similarly, Kontsevich (\cite[Problem 3.48]{kirby1995problems}) conjectured for compact $3$-manifold $M$ with  non-empty boundary, the classifying space $\BDiff(M, \text{rel }\partial)$ has a finite-dimensional model. This conjecture is known to hold for irreducible $3$-manifolds with non-empty boundary (\cite{MR1486644}). In this paper, we shall prove the homological finiteness of these classifying spaces for reducible $ 3$-manifolds with a condition on their boundary. 

Throughout this paper, for brevity, we write $\Diff(M,\text{rel }\partial)$ and $\tH(M, \text{rel }\partial) $ to denote the smooth orientation preserving diffeomorphisms and orientation preserving homeomorphisms respectively whose supports (i.e. the closure of points that are not fixed) are away from the boundary $\partial M$ so they are the identity near the boundary. In general, when we use $\text{rel }X$ in the diffeomorphism group, for some $X\subset M$, we mean those diffeomorphisms or homeomorphisms whose supports are away from $X$.

We say a path-connected space $K$ is   {\it strongly homologically finite} if for all $\bZ[\pi_1(K)]$-modules $A$ that are finitely generated as an abelian group,
$H_*(K; A)$ is finitely generated in each degree and is non-zero in finitely many degrees.
\begin{thm}\label{Kont1}Let $M$ be an orientable $3$-manifold that is a connected sum of compact irreducible $3$-manifolds that are not diffeomorphic to the $3$-ball and each have a nontrivial boundary.
 Then the classifying space $\BDiff(M, \text{rel }\partial)$ is strongly homologically finite.
 \end{thm} 
 In the irreducible case, the homotopy type of the group $\Diff(M)$ is very well studied.  When $M$ admits one of Thurston's geometries, there has been an encompassing program known as the generalized Smale's conjecture that relates the homotopy type of $\Diff(M)$ to the isometry group of the corresponding geometry (for more details and history, see the discussions in Problem 3.47 in \cite{kirby1995problems} and Sections 1.2 and 1.3 in \cite{MR2976322}). For $\bS^3$, it was proved by Hatcher (\cite{hatcher1983proof}), and for Haken $3$-manifolds, it is a consequence of Hatcher's work and also understanding the space of incompressible surfaces (\cite{MR224099, MR0420620, MR0448370}) inside such manifolds. Recently Bamler and Kleiner (\cite{bamler2019ricci, bamler2019diffeomorphisms}) used Ricci flow techniques to settle the generalized Smale's conjecture for all $3$-manifolds admitting the spherical geometry or in the Nil geometry.   Hence, this recent body of work using Ricci flow techniques addresses all cases of the generalized Smale's conjecture.
 
 Recall that a compact $3$-manifold $M$ (with or without boundary) is called prime if the existence of a diffeomorphism between $M$ and the connected sum $M_1\# M_2$ of two compact $3$-manifolds $M_1$ and $M_2$, implies that at least one of them is diffeomorphic to the $3$-sphere. The prime decomposition theorem says that every compact $ 3$ manifold  is diffeomorphic to the connected sum of prime manifolds.  A prime closed $3$-manifold is either diffeomorphic to $\bS^1\times \bS^2$ or it is irreducible (i.e. every embedded $\bS^2$ bounds a ball). On the other hand, geometric manifolds are the building blocks for irreducible manifolds. Given the generalized Smale's conjecture, we have a good understanding of the homotopy type of the diffeomorphism groups for these atomic pieces. The JSJ and geometric decomposition theorems (see \cite[Chapter 2, section 6]{neumann1996notes} for the statement of these theorems) give a way to cut an irreducible manifold along embedded tori into these building blocks. If the JSJ decomposition is non-trivial for an irreducible manifold, then it will be Haken whose diffeomorphism groups are well studied. Hence, given that we also know the homotopy type of the diffeomorphism group of $\bS^1\times \bS^2$ by Hatcher's theorem (\cite{hatcher1981diffeomorphism}), we have a good understanding of the homotopy type of diffeomorphism group of prime manifolds. In the reducible case, the prime decomposition theorem cuts the manifold along separating spheres into its prime factors. The difficulty, however, in understanding the reducible case is to relate the diffeomorphism group of a reducible manifold to the diffeomorphisms of its prime factors. 

 C\' esar de S\' a and Rourke (\cite{MR513752}) proposed to describe the homotopy type of $\Diff(M)$ in terms of the homotopy type of diffeomorphisms of the prime factors and an extra factor of the loop space on ``the space of prime decompositions''. Hendriks-Laudenbach (\cite{MR780734}) and  Hendriks-McCullough (\cite{MR893801}) found a  model for this extra factor. Later Hatcher, in an interesting unpublished note, proposed a {\it finite} dimensional model for this ``space of prime decompositions'' and more interestingly, he proposed that there should be a map between $\BDiff(M)$ and the product of classifying spaces of diffeomorphisms of all prime factors of $M$. He envisaged that this map sits in a fiber sequence whose fiber is ``the space of prime decompositions''.
 
 Hatcher's approach, if completed, would also solve Kontsevich's conjecture in the special case of reducible $ 3$-manifolds such that all the irreducible factors have non-empty boundaries. So our result is {\it the homological version} of what Hatcher intended to prove about Kontsevich's conjecture. However, instead of trying to build the map, we take the geometric group theory approach by letting the abstract group of diffeomorphisms act on a ``huge" simplicial complex inspired by the techniques that  Kathryn Mann and the author (\cite{mann2020dynamical}) used to study the second homology of $\BDiff(M)$.

For technical simplicity, we work with the homeomorphism groups instead of diffeomorphism groups. This does not make a difference to the main result. The reason is that Cerf (\cite{cerf1961topologie}) assumed Smale's conjecture which was later proved by Hatcher (\cite{hatcher1983proof}) to show that in these low dimensions, the inclusion $\Diff(M)\hookrightarrow \Homeo(M)$ is, in fact, a weak homotopy equivalence. 

On the other hand, in all dimensions, by Mather-Thurston's theorem (\cite[Corollary (b) of theorem 5]{thurston1974foliations} and see \cite[Section 2, Theorem 2.5]{mcduff1980homology} for the proof) for homeomorphisms, we have the natural map 
\begin{equation}\label{Thurston}
\BdD(M)\to \BD(M),
\end{equation}
which is an acyclic map and in particular it induces a homology isomorphism in all degrees. The same statement also holds for manifolds with boundary and in the relative case in particular relative to the boundary when it is non-empty (see \cite[Section 2, Theorem 2.5]{mcduff1980homology}).

 Hence to prove the main theorem, we use a homological approach where we consider the action of $\dD(M, \text{rel }\partial)$ on a simplicial complex $\mathcal{S}(M)$ given by {\it the complex of essential spheres}, to give a model for $\BdD(M, \text{rel }\partial)$ suitable for an inductive argument to prove the main theorem. 
 
The revision process of this paper, which first appeared on the arXiv in 2021, was prolonged due to the author's engagement in other research projects, as well as periods of neglect and lack of motivation. In the meantime, Boyd, Bregman, and Steinebrunner have obtained a complete resolution of the problem in full generality (\cite{boyd2024moduli}). Nevertheless, the author hopes that the homological approach presented in this shorter paper remains of independent interest.


\subsection*{Acknowledgment}The author was partially supported by NSF grants DMS-2113828 and CAREER Grant DMS-2239106, a grant from the Simons Foundation (41000919, SN), and the European Research Council (ERC) under the European Union’s Horizon 2020 research and innovation program (grant agreement No. 682922). He thanks Sander Kupers and Andrea Bianchi for their helpful comments on an earlier draft of this paper. The author is also very grateful to the referee for numerous valuable suggestions and constructive criticism that significantly improved the clarity and accuracy of the manuscript. Despite the author's delay over the years in revising the paper, the referee's exceptional dedication and thoughtful feedback were both helpful and encouraging in bringing the rough first draft to its published form.

\section{Sphere complexes} In this section, we assume that $M$ is a compact reducible $3$-manifold with a non-empty boundary. Additionally, in this section, we assume that we do not have spherical boundary components in order to have a prime decomposition with no $3$-disk factor  (\cite[Chapter 3, Lemma 3.7]{MR0415619}). To study the homological finiteness of $\BdD(M, \text{rel }\partial)$ inductively based on the number of prime factors in the prime decomposition of $M$, we shall first construct a  simplicial complex $\mathcal{S}(M)$ on which $\dD(M, \text{rel }\partial)$ acts simplicially.

\begin{defn}Let $\mathcal{S}(M)$ be a simplicial complex whose vertices are given by locally flat embeddings $\phi\colon \bS^2\hookrightarrow M$ whose images are essential spheres i.e. $\phi$ is not null-homotopic  and simplices in $\mathcal{S}(M)$ are given by collections of locally flat embeddings whose images are disjoint.
\end{defn}
\begin{rem}
For $3$-manifolds that do not have $\bS^1\times \bS^2$ factors and have no spherical boundary components, essential spheres are the same as separating spheres.
\end{rem}
\begin{prop}\label{spherecomplex}
The simplicial complex  $\mathcal{S}(M)$ is contractible.
\end{prop}
\begin{proof}
In \cite[Lemma 4.3]{nariman2020local}, the author proved that the subcomplex of separating spheres in $M$ is contractible. If $M$ does not have $\bS^1\times \bS^2$ summands, which is the case that in fact, we are interested in, the complex of essential spheres $\mathcal{S}(M)$ is the same as the complex in \cite[Lemma 4.3]{nariman2020local}.  But the same proof shows that the complex $\mathcal{S}(M)$ is also contractible even when the prime decomposition of $M$ has  $\bS^1\times \bS^2$ summands. 
\end{proof} 


Note that the group $\dD(M, \text{rel }\partial)$ acts on $\mathcal{S}(M)$ simplicially. The complex $\mathcal{S}(M)$ is contractible by \Cref{spherecomplex}. Therefore, the homotopy quotient $\mathcal{S}(M)\hcoker \dD(M, \text{rel }\partial)$ is homotopy equivalent to $\BdD(M, \text{rel }\partial)$. The stabilizer of each simplex in $\mathcal{S}(M)$ is the subgroup of $\dD(M, \text{rel }\partial)$ that fixes a set of essential spheres pointwise so it is isomorphic to the homeomorphism group of a $3$-manifold whose connected components have fewer prime factors. But one issue is that $\mathcal{S}(M)$ has simplices of arbitrary large dimensions since we allow parallel spheres. To account for this infinite dimensionality, we use the simplicial complex that Hatcher and McCullough defined in \cite[Section 1]{hatcher1990finite}.
\begin{defn}
Let $[\mathcal{S}](M)$ be the simplicial complex whose vertices are the isotopy classes of essential embedded spheres in $M$. A set of vertices $\{[S_0], [S_1],\dots, [S_n]\}$ constitutes an $n$-simplex if there are pairwise disjoint embedded spheres $S_i'$ in $M$ such that for each $i$, the isotopy class of the sphere $S_i'$ is the class $[S_i]$. 
\end{defn} 
The mapping class group $\text{Mod}(M, \text{rel }\partial)=\pi_0(\tH(M, \text{rel }\partial))$ acts on $[\mathcal{S}](M)$ simplicially. Because all prime factors have non-trivial boundary components that are fixed by $\text{Mod}(M, \text{rel }\partial)$, as we shall see in \Cref{order} if the action of $\text{Mod}(M, \text{rel }\partial)$ fixes a simplex set-wise, it also fixes it pointwise. The complex $[\mathcal{S}](M)$ is finite-dimensional and also by Hatcher-McCullough's theorem (\cite[Proposition 2.2]{hatcher1990finite}) the set of the orbits of the action of $\text{Mod}(M, \text{rel }\partial)$ on simplices is also finite. 

To briefly recall why this is the case, they use a theorem of Scharlemann (see \cite[Appendix A, Lemma A.1]{Bonahon}) to find a ``normal" representative of each orbit.
Let the prime decomposition of $M$ be given by $P_1\# P_2\#\cdots \#P_r\#(\#^{g}\bS^1\times \bS^2)$ where $g$ summands are homeomorphic to $\bS^1\times \bS^2$.  Let $B$ be a punctured $3$-cell having ordered $r+2g$ boundary components so that $M$ is obtained by gluing $P_i\backslash \text{int}(\bD^3)$ to $i$-th sphere boundaries for $1\leq i\leq r$ and $g$ copies of $\bS^2\times [0,1]$ are glued along the remaining $2g$ boundary components (see \cite[Appendix A, Lemma A.1]{Bonahon} for more details). 
 \begin{figure}[h]
\[
\begin{tikzpicture}[scale=.2]

\draw[line width=1.05pt]  (5,0) arc (0:45: 5);
\draw[line width=1.05pt]  (1.29,4.83) arc (75:105: 5);

\draw[line width=1.05pt]  (5,0) arc (0:-108: 5);
\draw[line width=1.05pt]  (-5,0) arc (180:225: 5);
\draw[line width=1.05pt]  (-5,0) arc (180:135: 5);

\draw[line width=1.05pt] [dashed] (5,0) arc (10: 170: 5 and 1);
\draw[line width=1.05pt]  (5,0) arc (-10: -170: 5 and 1);

\draw[line width=1.05pt] (3,2.5).. controls (2,2) and (1.5, 2.5).. (1,3.2);
\draw[line width=1.05pt] [dashed] (3,2.5).. controls (2,3.2) and (1.5, 3.8).. (1,3.2);

\draw[line width=1.05pt] (-3,2.5).. controls (-2,2) and (-1.5, 2.5).. (-1,3.2);
\draw[line width=1.05pt] [dashed] (-3,2.5).. controls (-2,3.2) and (-1.5, 3.8).. (-1,3.2);

\draw[line width=1.05pt]  (-3,-2.5).. controls (-2,-2) and (-1.5, -2.5).. (-1,-3.2);
\draw[line width=1.05pt] [dashed] (-3,-2.5).. controls (-2,-3.2) and (-1.5, -3.8).. (-1,-3.2);

\draw [line width=1.05pt] (3,2.5) .. controls (7,10) and (4,12)  .. (1,3.2);
\draw [line width=1.05pt] (-3,2.5) .. controls (-7,10) and (-4,12)  .. (-1,3.2);
\draw [line width=1.05pt] (-3,-2.5) .. controls (-7,-10) and (-4,-12)  .. (-1,-3.2);

\begin{scope}[shift={(-6.8,7)}, scale=0.5]
\node at (6.4,-0.8) {{\Small$P_1$}};

\end{scope}

\begin{scope}[shift={(0.3,6)}, scale=0.5]
\node at (5.9,0.5) {{\Small$P_2$}};

\end{scope}

\begin{scope}[shift={(-6.8,-7)}, scale=0.5]
\node at (5.9,0.5) {{\Small$P_3$}};
\end{scope}

\node at (2,1.5) {$$};
\node at (-2,1.5) {$$};
\node at (-2,-1.5) {$$};
\node at (2,-2.5) {$B$};
\end{tikzpicture}
\]
\caption{$\sigma$ here is a $2$-simplex consisting of 3 separating spheres that are drawn in one dimension lower. }\label{sc}
\end{figure}
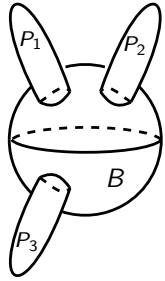

 \begin{lem}[Scharlemann]\label{Sch}
 For any simplex $\sigma \subset \mathcal{S}(M)$, there is a homeomorphism $f$ such that $f(\sigma)\subset B$. 
 \end{lem} 

Now as Hatcher and McCullough observed in \cite[Proposition 2.2]{hatcher1990finite}, there are finitely many isotopy classes of essential spheres in $B$ since they are determined by the way they partition the boundary components of $B$. This observation implies the finiteness of the orbits of the action of $\text{Mod}(M, \text{rel }\partial)$ on simplices of $[\mathcal{S}](M)$. 

 
The skeletal filtration on $[\mathcal{S}](M)$ induces a filtration on the quotient space 
\begin{equation}\label{filt}
\mathcal{F}_0\subset  \mathcal{F}_1\subset \dots \subset \mathcal{F}_n=[\mathcal{S}](M)/\text{Mod}(M, \text{rel }\partial),
\end{equation}
and by Hatcher and McCullough's observation, the filtration quotients are given by the wedge of a {\it finite} number of spheres. The reason that the filtration quotients are spheres is the fact that if the action fixes a simplex set-wise, then it fixes it pointwise. Let $\mathcal{O}_p$ be the set of orbits of the action of $\text{Mod}(M, \text{rel }\partial)$ on $p$-simplices of $[\mathcal{S}](M)$. So $\mathcal{F}_p-\mathcal{F}_{p-1}$ is homeomorphic to $\coprod_{\sigma\in \mathcal{O}_p} \dot{\Delta}^p_{\sigma}$, the disjoint union of open $p$-simplices indexed by $\mathcal{O}_p$. 

The natural simplicial map $\mathcal{S}(M)\to [\mathcal{S}](M)$ that sends spheres to their isotopy classes is equivariant with respect to the map $\dD(M, \text{rel }\partial)\to \text{Mod}(M, \text{rel }\partial)$. So we have a map $\mathcal{S}(M)/\dD(M, \text{rel }\partial)\to [\mathcal{S}](M)/\text{Mod}(M, \text{rel }\partial)$ which in turn induces a map
\[
\eta\colon \mathcal{S}(M)\hcoker \dD(M, \text{rel }\partial)\to [\mathcal{S}](M)/\text{Mod}(M, \text{rel }\partial).
\]
\begin{defn}
Let $\mathcal{L}$ be a local coefficient system, on a space $X$. We say $X$ is $\mathcal{L}$-{\it homologically finite}, if $H_*(X; \mathcal{L})$ is finitely generated in each degree and is non-zero in finitely many degrees.
\end{defn}

We are interested in local coefficient systems that are pullbacks of local coefficients systems on $\BD(M, \text{rel }\partial)$. Since we have the maps
\[
\mathcal{S}(M)\hcoker \dD(M, \text{rel }\partial)\to \mathrm{B}\dD(M, \text{rel }\partial)\to \BD(M, \text{rel }\partial),
\]
for each point $x$ in the image of $\eta$ we have a map $\eta^{-1}(x)\to  \BD(M, \text{rel }\partial)$ that is uniquely determined up to homotopy. It turns out, as we shall see, that the restriction of $\eta$ to each open simplex $\dot{\Delta}^p_{\sigma}$ is a fiber bundle (see \Cref{last}) whose fiber, by an inductive argument, is $\mathcal{L}$-homologically finite for all $\mathcal{L}$ that are pullbacks from local coefficient systems on $\BD(M, \text{rel }\partial)$.
\begin{defn}
Given a map $f\colon X\to \BD(M, \text{rel }\partial)$, we say that $(X,f)$ is {\it weakly homologically finite}, if $X$ is $\mathcal{L}$-homologically finite for all $\mathcal{L}$ that are pullbacks from local coefficient systems on $\BD(M, \text{rel }\partial)$ and are finitely generated abeliean groups at each point. In cases where we consider a space $X$, the homotopy class of the map $f$ is understood from the context, in which case we drop the map and say $X$ is weak homologically finite.
\end{defn}
%
\begin{thm}\label{main'}
The preimage $\eta^{-1}(x)$ is weak homologically finite  for all $x\in [\mathcal{S}](M)/\text{Mod}(M, \text{rel }\partial)$.
\end{thm}

Then the weak homological finiteness of $\BdD(M, \text{rel }\partial)$ will follow from a general statement about simplicial complexes where in this generality, to clarify the original argument, was suggested to the author by the referee.

\begin{lem}\label{mainlem} Let $G$ be a discrete group, and let $X$ and $Y$ be two $G$-simplicial complexes. Let $X\to Y$ be a $G$-equivariant map of simplicial complexes. This map induces the map
$\eta\colon X\hcoker G \to Y/G$. Let $\mathcal{L}$ be a local coefficient system on $X\hcoker G$. Suppose the following conditions hold
\begin{enumerate}
\item if $G$ fixes a simplex as a subset in $X$ or $Y$ then it fixes it pointwise,
\item $\eta^{-1}(y)$ is $\mathcal{L}_y$-homologically finite  for all $y \in Y/G$ where $\mathcal{L}_y$ is the pullback of $\mathcal{L}$,
\item There are finitely many orbits of the action of $G$ on simplices of $Y$ and as a result $Y/G$ is a finite CW complex,
\end{enumerate}
then $X\hcoker G$ is also $\mathcal{L}$-homologically finite.
\end{lem}
The simplicial map $ \mathcal{S}(M)\to [\mathcal{S}](M)$ is $\dD(M, \text{rel }\partial)$-equivariant and we already know that the conditions (1) and (3) are satisfied. So \Cref{Kont1} follows from \Cref{mainlem} and \Cref{main'}. The bulk of the work is to prove \Cref{main'} and we shall prove \Cref{mainlem} in \Cref{last}.

In the \Cref{fiber} and \Cref{barconstruction}, we find a model for $\eta^{-1}(x)$ to which we can apply the induction hypothesis (i.e.~the strong homological finiteness of $\BD(M, \text{rel }\partial)$ for $M$ with fewer prime factors) when $M$ is connected sum of irreducible factors that each have a non-trivial boundary.

 \section{Reformulation of the main theorem and an inductive strategy}
 In this section, we shall use the hypothesis that the prime decomposition of $M$ consists of irreducible factors that each have non-empty non-spherical boundary components. We choose a base point on the boundary of $M$. We denote this boundary component by $\partial_*M$. The goal is for each $p$ and each $x\in\mathcal{F}_p-\mathcal{F}_{p-1}$ to find a semi-simplicial space $X_{\bullet}$ whose realization admits an acyclic map from $\eta^{-1}(x)$ and sits in a fibration sequence so that by induction on the number of prime factors we could argue that the fiber and the base are weakly homologically finite with compatible maps. 
 
 The advantage of working with $ 3$-manifolds that are connected sums of irreducible pieces such that each have a non-empty boundary is: 
 \begin{itemize}
\item When we cut along essential separating spheres, the remaining pieces each have a non-spherical boundary component that is fixed and we shall use this for the inductive argument.
\item Since each irreducible factor has a non-trivial boundary that is fixed, homeomorphic irreducible factors cannot be permuted under the action of $\tH(M, \text{rel }\partial)$.
 \end{itemize}
 
 To be able to induct on the number of prime factors of $M$, we shall prove a slightly more general statement of \Cref{Kont1} taking into account spherical boundary components. 
 
Recall that the corresponding automorphism groups in $C^0$-category and $C^{\infty}$-category are weakly homotopy equivalent.
 \begin{thm}[Cerf and Hatcher]\label{CH}
 For a compact $3$ manifold $M$, the inclusion $\Diff(M)\to \tH(M)$ is a weak homotopy equivalence.
 \end{thm}
Cerf (\cite{cerf1961topologie}) assumed Smale's conjecture which was later proved by Hatcher (\cite{hatcher1983proof}) to show that in these low dimensions, the inclusion $\Diff(M)\hookrightarrow \Homeo(M)$ is a weak homotopy equivalence. So to prove our homological finiteness result, we freely use diffeomorphism groups and homeomorphism groups depending on the convenience of the situation. 
 
 Let $e_i\colon \bD^3\hookrightarrow M$ for $1\leq i\leq k+l$ be disjoint embeddings and let $N$ be the $3$-manifold obtained from $M$ by removing  $e_i(\text{int}(\bD^3))$ for all $1\leq i\leq k+l$. So the boundary of $N$ is the union of $\partial M$ with sphere boundary components, which we denote by $S_i$. We denote the union of the sphere boundary components $ \{S_i\}_{i=1}^k$ by $S_{\text{free}}$ and union of the rest of the sphere boundary components by $S_{\text{fixed}}$. Let $\tH(N, S_{\text{free}}, \text{rel }(\partial M\cup S_{\text{fixed}}))$ be the subgroup of $\tH(N, \text{rel }(\partial M\cup S_{\text{fixed}}))$ whose elements  fix each sphere in $S_{\text{free}}$ set-wise. 
 
 \begin{thm}\label{Kont}
 Then $\BD(N, S_{\text{free}}, \text{rel }(\partial M\cup S_{\text{fixed}}))$ is strongly homologically finite.
 \end{thm}
 
 \Cref{Kont1} is a special case of \Cref{Kont}. But as we shall see in \Cref{sbc}, in fact, they are equivalent statements. However, the statement of  \Cref{Kont} is more convenient for the inductive argument.

Our first goal is to use \Cref{Kont} inductively for fewer prime factors than the number of prime factors of $M$ to show that for each $p$ and each $x\in\mathcal{F}_p-\mathcal{F}_{p-1}$, the pre-image $\eta^{-1}(x)$ is weakly homologically finite. To fix ideas, let $x$ be in $\mathcal{F}_0$ in the filtration \ref{filt} which is the image of a separating sphere $S\subset M$. This is because all vertices in $\mathcal{S}(M)$ are separating since the prime decomposition of $M$ does not have $\bS^1\times \bS^2$ summands.  
 
 Suppose that the sphere $S$ cuts the manifold $M$ into two pieces $M_1$ and $M_2$ where $M_1$ contains $\partial_* M$, the boundary component of $M$ with the base point. Let $\tH(M_1, S, \text{rel }\partial M)$ be the subgroup of $\tH(M_1)$ that fixes the boundary component $S$ set-wise and the rest of the boundary components pointwise. In \Cref{fiber}, we shall prove that there is an acyclic map from $\eta^{-1}(x)$  to the realization of a semi-simplicial space (in fact a two-sided bar construction) $X_{\bullet}$ that fits in a homotopy fiber sequence
\begin{equation}\label{fib1}
\BD(M_2, \text{rel }\partial)\to ||X_{\bullet}||\to \BD(M_1, S, \text{rel }\partial M).
\end{equation}
Note that $M_1$ and $M_2$ have sphere boundary components. When we fill in those sphere boundary components with balls, they have fewer prime factors compared to $M$. By the induction hypothesis, \Cref{Kont} implies that the fiber $\BD(M_2, \text{rel }\partial)$ and the base $\BD(M_1, S, \text{rel }\partial M)$ are strongly homologically finite. Then the following lemma implies that $||X_{\bullet}||$ is also strongly homologically finite.
\begin{lem}\label{kupersfinitenesslemma} Let $F\xrightarrow{i} E\to B$ be a fiber sequence where $B$ is path-connect.  If $B$ is strongly homologically finite, $E$ is path-connected and $H_*(F;
i^*(M))$ is finitely
generated  in each degree and nonzero in finitely many degrees for all $\bZ[\pi_1(E)]$-modules $M$ that are finitely generated 
as abelian groups, then $E$ is strongly homologically finite.
\end{lem}
\begin{proof}
The proof is the same as the proof of \cite[Lemma 2.5 (ii)]{kupers2019some}. Kupers' definition of being homologically finite only requires homology with finitely generated abelian local coefficients to be finitely generated in each degree. We further require that there are only finitely many non-zero homology groups with such local coefficients. But the same proof in \cite[Lemma 2.5 (ii)]{kupers2019some} works verbatim for our definition too.
\end{proof}
It will be by construction that the map $\eta^{-1}(x)\to \BD(M, \text{rel }\partial)$ factors through the map $\eta^{-1}(x)\to ||X_{\bullet}||$. Therefore, in this way, we shall deduce that $\eta^{-1}(x)$ is a weakly homologically finite space by induction on the number of prime factors and then \Cref{Kont} will follow from \Cref{mainlem}.  First, let us settle the induction's base case for \Cref{Kont}.

  \section{The base case of induction and spherical boundary components}\label{sbc}
In this section, we shall see that \Cref{Kont1} and \Cref{Kont} are equivalent and we prove \Cref{R} that is stronger than the base case of \Cref{Kont}.

 Let $P$ be an irreducible $ 3$-manifold with a non-empty and non-spherical boundary. Let $e_i\colon \bD^3\hookrightarrow P$ for $1\leq i\leq k+l$ be disjoint embeddings and let $N$ be the $3$-manifold obtained from $P$ by removing  $e_i(\text{int}(\bD^3))$ for all $1\leq i\leq k+l$. So the boundary of $N$ is the union of $\partial P$ with sphere boundary components, which we denote by $S_i$. We denote the union of the sphere boundary components $ \{S_i\}_{i=1}^k$ by $S_{\text{free}}$ and union of the rest of the sphere boundary components by $S_{\text{fixed}}$. The group $\tH(N, \text{rel }(\partial P\cup S_{\text{fixed}}))$  fixes the boundary components $S_{\text{free}}$ set-wise. To keep track of free spherical boundary components, in what follows, we also add $S_{\text{free}}$ to the notation $\tH(N, \text{rel }(\partial P\cup S_{\text{fixed}}))$ (e.g.  $\tH(N, S_{\text{free}}, \text{rel }(\partial P\cup S_{\text{fixed}}))$ without having ``rel'' before the free boundary components).
 
 \begin{thm}\label{R} Then $\BD(N, S_{\text{free}}, \text{rel }(\partial P\cup S_{\text{fixed}}))$ has a finite CW complex model. 
 \end{thm}
This of course implies the base case of \Cref{Kont}. Since the corresponding diffeomorphism groups and homeomorphism groups are weakly equivalent (by \Cref{CH}), we shall instead prove that $\BDiff(N, S_{\text{free}}, \text{rel }\partial P\cup S_{\text{fixed}})$ has a finite CW complex model. 
 
 We already know the homotopical finiteness for an irreducible $3$-manifold with a non-empty boundary (\cite{MR1486644}).
\begin{thm}[Hatcher-McCullough]\label{HM} If $M$ is an irreducible $3$-manifold with a non-empty boundary, then $\BDiff(M,\text{rel }\partial)$ has the homotopy type of a finite CW-complex.
\end{thm}
So we know that $\BDiff(P, \text{rel }\partial)$ has a finite CW complex model. We want to inductively fix $e_i(\bD^3)$ either set-wise or pointwise and still get a finite CW complex model. 
\begin{lem}\label{frame} Suppose $P$ is a $3$-manifold with possibly nonempty boundary. Let $\partial_1$ be a subset of boundary components containing the non-spherical components (it could also contain spherical boundary components) and let $S_{\text{free}}$ be the union of remaining spherical components. Let $e\colon D^3\hookrightarrow P$ be an embedding of a ball inside $P$. If $\BDiff(P, S_{\text{free}}, \text{rel }\partial_1)$ has the homotopy type of a finite CW-complex, so does $\BDiff(P, S_{\text{free}}, \text{rel }\partial_1 \cup e(D^3))$. Similarly, if $\BDiff(P, S_{\text{free}}, \text{rel }\partial_1)$ is strongly homologically finite , so is $\BDiff(P, S_{\text{free}}, \text{rel }\partial_1 \cup e(D^3))$.
\end{lem}
\begin{proof} Including the diffeomorphism group that fixes a neighborhood of the boundary into the diffeomorphism group that fixes the boundary pointwise is a weak homotopy equivalence (see \cite[Chapter 4 and 5]{kupers2019lectures}). So let $\Diff(P, S_{\text{free}}, \text{fix }e(D^3), \text{rel }\partial_1)$ be the subgroup of $\Diff(P, S_{\text{free}}, \text{rel }\partial_1)$ that fixes $e(D^3)$ pointwise. The inclusion
\[
\Diff(P, S_{\text{free}}, \text{rel }\partial_1 \cup e(D^3))\to \Diff(P, S_{\text{free}}, \text{fix }e(D^3), \text{rel }\partial_1),
\]
is a weak homotopy equivalence. Hence,  using \cite[Theorem C]{MR0123338} and the above weak equivalence, we have a homotopy fiber sequence
\[
\Diff(P, S_{\text{free}}, \text{rel }\partial_1 \cup e(D^3))\to \Diff(P, S_{\text{free}}, \text{rel }\partial_1)\to \text{Emb}^+(D^3, P)\simeq \text{Fr}^+(P),
\]
where $\text{Emb}^+(D^3, P)$ is the space of orientation preserving embeddings and $\text{Fr}^+(P)$ is the oriented frame bundle of $M$. It can be delooped (similar to the standard fact \cite[Proposition 1.80]{felix2008algebraic}) to induce the fiber sequence
\[
\text{Fr}^+(P)\to \BDiff(P, S_{\text{free}}, \text{rel }\partial_1 \cup e(D^3))\to \BDiff(P, S_{\text{free}}, \text{rel }\partial_1).
\]
The base and the fiber of this fiber sequence have the homotopy type of a finite CW-complex (strongly homologically finite). Therefore, the total space also has a finite CW-complex model (strongly homologically finite).
\end{proof}
\begin{lem}\label{markpoint} Let $P$ be as in the previous lemma and $e\colon D^3\hookrightarrow P$ be an embedding of a ball inside $P$. Let $x\in P$ be the image of the center of the ball. Let $M$ be the manifold obtained from $P$ by removing $\text{int}(e(D^3))$ so it has a sphere boundary $S$. Let the group $\Diff(P, x, S_{\text{free}},\text{rel }\partial_1)$ be the subgroup of $\Diff(P, S_{\text{free}},\text{rel }\partial_1)$ that fixes $x$ and $\Diff(M, S_{\text{free}}, \text{rel }\partial_1)$ be the subgroup of $\Diff(M)$ that is the identity near the boundary components $\partial _1$ and fixes each of the other boundary components set-wise. Then there is a zig-zag of group homomorphisms that are homotopy equivalences between the group $\Diff(P, x, S_{\text{free}},\text{rel }\partial_1)$ and the group $\Diff(M, S_{\text{free}}, \text{rel }\partial_1)$.
\end{lem}
%
\begin{proof} For simplicity we consider the case where P is closed, the general case
follows similarly. Let $\Diff(M, \text{rel } \partial_{\text{SO}(3)})$ be the subgroup of $\Diff(M)$ that on a neighborhood of the boundary $S$ restricts to the subgroup of rigid rotations in the following sense. We fix a collar neighborhood $e\colon S^2\times [0,1)\hookrightarrow M$ extending the parametrization of the boundary $S$. The group of rotations $\text{SO}(3)$ acts on this collar neighborhood by acting on each slice $e(S^2\times\{t\})$ by a fixed rotation. For each element $f$ of $\Diff(M, \text{rel } \partial_{\text{SO}(3)})$, there exists a positive $\epsilon$ such that the restriction of $f$ to $e(S^2\times [0,\epsilon))$ is the same as the action of $\text{SO}(3)$. Recall that Smale's theorem (\cite{MR0112149}) implies that $\Diff_0(S^2)\simeq \text{SO}(3)$. Hence by the comparison of fiber sequences, it is easy to see that $\Diff(M, \text{rel } \partial_{\text{SO}(3)})$ is homotopy equivalent to $\Diff(M)$. 

So it is enough to show that the natural inclusion 
\[
\Diff(M, \text{rel } \partial_{\text{SO}(3)})\to \Diff(P, x),
\]
that is induced by extending a rotation on the boundary of $e(D^3)$ to its interior, is a homotopy equivalence.

Recall that Hatcher's theorem implies that $\Diff(D^3, \text{rel }S^2)$ is contractible which in turn implies that the restriction map $\Diff(D^3)\to \Diff(S^2)$ is a homotopy equivalence. Let  $\Diff(P, e(D^3))$ be the subgroup of $\Diff(P)$ that fixes $e(D^3)$ set-wise. Since $\Diff^+(D^3)\simeq \text{SO}(3)$, by the comparison of fibrations, one can see that the inclusion $\Diff(M, \text{rel } \partial_{\text{SO}(3)})\to\Diff(P, e(D^3))$ is a homotopy equivalence. 

On the other hand, since $D^3$ is contractible, the fiber sequence obtained by the action of $\Diff^+(D^3)$ on $D^3$ implies that the inclusion $\Diff^+(D^3, 0)\hookrightarrow \Diff^+(D^3)$ is a weak equivalence where $\Diff^+(D^3, 0)$ is the subgroup fixing the origin of $D^3$. Let $\Diff(P, e(D^3), x)$ be the subgroup of $\Diff(P, e(D^3))$ that fixes the point $x$. By comparison of fiber sequences, one can see that $\Diff(P, e(D^3), x)\to\Diff(P, e(D^3))$ is also a weak equivalence. Therefore, it is enough to show that $\Diff(P, e(D^3), x)\to \Diff(P,x)$ is a weak equivalence. Consider  the following fiber sequence that is a variant of the fiber sequence  \cite[Theorem C]{MR0123338}
\[
\Diff(P, e(D^3), x)\to \Diff(P, x)\to \text{Emb}^+((D^3,0), (P, x))/\Diff(D^3,0),
\]
where $\text{Emb}^+((D^3,0), (P, x))/\Diff(D^3,0)$ is the space of unparametrized smooth embeddings of $D^3$ that send its center to $x$. It is easy to see that $\text{Emb}^+((D^3,0), (P, x))/\Diff(D^3,0)$ is contractible. Hence, $\Diff(P, e(D^3), x)\to \Diff(P, x)$ is a weak equivalence.
\end{proof}
%

 \begin{proof}[Proof of \Cref{R}] We shall prove that $\BDiff(N, S_{\text{free}}, \text{rel }(\partial P\cup S_{\text{fixed}}))$ has a finite CW complex model. Let $M$ be the manifold obtained from $P$ by removing $e_i(\text{int}(\bD^3))$ for all $1\leq i\leq k$. Given \Cref{frame}, it is enough to prove that $\BDiff(M, S_{\text{free}}, \text{rel }\partial P)$ has a finite CW complex model. 
 
 Let $x_i$ be a point in $P$ given by the image of the center of the ball $e_i(\text{int}(\bD^3))$. Let $\Diff(P,\{x_1,\dots,x_k\}, \text{rel }\partial P)$ be the subgroup of $\Diff(P, \text{rel }\partial P)$ consisting of those elements that fix each $x_i$. 
 
 \begin{claim*}The classifying space $\BDiff(M, S_{\text{free}}, \text{rel }\partial P)$ is homotopy equivalent to $\BDiff(P,\{x_1,\dots,x_k\}, \text{rel }\partial P)$.
\end{claim*}
This is easily deduced by applying \Cref{markpoint} inductively $k$ times.  
%

Let $\text{PConf}_k(P)$ be the space of ordered configuration space of $k$ points in the interior of $P$. Forgetting the first point induces the map
\[
\text{PConf}_k(P)\to \text{PConf}_{k-1}(P)
\]
which is a fibration whose fiber is homotopy equivalent to the complement of $(k-1)$ points in $P$.  Hence, inductively we conclude that $\text{PConf}_k(P)$ is weakly equivalent to a finite CW complex. Considering the action of $\Diff(P, \text{rel }\partial P)$ on the $k$ points $\{x_1,\dots,x_k\}$ gives a Palais fiber sequence (\cite[Theorem C]{MR0123338}). By delooping this fiber sequence, we have a fiber sequence
\[
\text{PConf}_k(P)\to \BDiff(P,\{x_1,\dots,x_k\}, \text{rel }\partial P)\to \BDiff(P, \text{rel }\partial P).
\]
Since both $\BDiff(P, \text{rel }\partial P)$ and $\text{PConf}_k(P)$ have a finite CW complex model so does $ \BDiff(P,\{x_1,\dots,x_k\}, \text{rel }\partial P)$.
 \end{proof}
 Note that proof of \Cref{markpoint} and the proof of the claim above implies that if $\BDiff(M, \text{rel }\partial)$ is homologically finite, so is $\BDiff(N, S_{\text{free}}, \text{rel }(\partial P\cup S_{\text{fixed}}))$. Therefore, \Cref{Kont} is also implied by \Cref{Kont1}. 

 \section{Proof of the technical \Cref{mainlem}}\label{last}We denote the image of a simplex $\Delta\subset Y$ in $Y/G$ by $[\Delta]$. Because of conditions (2) and (3), the cells $[\Delta]$ give a finite CW structure on $Y/G$ and let $S$ be the finite set of all cells. 
 
 \noindent{\bf Step $0$:} For a simplex $\Delta\subset Y$, let $X(\Delta)$ be the subcomplex of $X$ that is the preimage of $\Delta$ under the map $X\to Y$. Then it is well known (\cite[Lemma 1.7 in Essay III at page 94]{MR0645390}) that the restriction of $X(\Delta)\to \Delta$ to the interior $\dot{\Delta}$ is a trivial fiber bundle. A more streamlined proof is given in \cite[Theorem 1.3.1]{MR0184242}. Note that the finiteness condition in \cite{MR0184242, MR0645390} is used in the second half of their proof where they want to prove it is a PL fiber bundle. To prove that topologically this restriction is a trivial fiber bundle as they showed, the finiteness condition on $X(\Delta)$ is not needed. 
 
 \noindent{\bf Step $1$:} Let $[\dot{\Delta}_{\alpha}]$ be the interior of the cell $[\Delta_{\alpha}]$ for $\alpha \in S$. We shall prove that for each cell $[\Delta_{\alpha}]$, the restriction of $\eta$ to the interior $[\dot{\Delta}_{\alpha}]$ is a trivial fiber bundle. 
 
 Let $X([\Delta_{\alpha}])\subset X$ be the subcomplex that is the pre-image of $[\Delta_{\alpha}]$ under the map $g$ that is the composition $g\colon X\to Y\to Y/G$ and let $X([\dot{\Delta}_{\alpha}])$ be the preimage of $[\dot{\Delta}_{\alpha}]$. So we want to show that the map
 \[
 X([\dot{\Delta}_{\alpha}])\hcoker G\to [\dot{\Delta}_{\alpha}],
\]
 is a trivial fiber bundle. Since the map $X\to Y$ is $G$-equivariant, the map $$X([\dot{\Delta}_{\alpha}])\to \text{orbit}(\dot{\Delta}_{\alpha})$$ is also a trivial fiber bundle  by step $0$ where $\text{orbit}(\dot{\Delta}_{\alpha})$ is the orbit of $\dot{\Delta}_{\alpha}$ under the $G$ action. Let $b_{\alpha}$ be the barycenter of $[\dot{\Delta}_{\alpha}]$. So there is a natural homeomorphism $X([\dot{\Delta}_{\alpha}])\cong  [\dot{\Delta}_{\alpha}]\times g^{-1}(b_{\alpha})$ which is $G$-equivariant where $G$-action on   $[\dot{\Delta}_{\alpha}]$ is trivial and on $X([\dot{\Delta}_{\alpha}])$ and $g^{-1}(b_{\alpha})$ are the natural actions restricted from the action on $X$. Hence, we obtain a natural homeomorphism
 \[
 X([\dot{\Delta}_{\alpha}])\hcoker G\cong [\dot{\Delta}_{\alpha}]\times (g^{-1}(b_{\alpha})\hcoker G),
 \]
 that commutes with the projection to $ [\dot{\Delta}_{\alpha}]$.

 \noindent{\bf Step 2:} Since the inclusion of sub-CW-complexes are cofibrations, there is an open neighborhood $U$ of each sub-CW-complex $K$ in $Y/G$ that deformation retracts to $K$. But we also want $U$ to satisfy the following property. Let $Y(K)$ be the $G$-invariant sub-simplicial-complex of $Y$ that is the preimage of $K$ via the quotient map $Y\to Y/G$ and let $Y(U)$  be a $G$-invariant open neighborhood of $Y(K)$ in the preimage of $U$ in $Y$. Similarly, we have  $G$-invariant subspaces $X(K)$ and $X(U)$. Since the map $X\to Y$ is simplicial, we want to choose $U$ small enough so that $X(U)$ deformation retracts to $X(K)$. 
 
 To fix such a neighborhood $U$, we can choose it combinatorially (one can also do it by putting metrics on $Y$ and $X$). Barycentrically subdivide $Y$ twice and choose the regular neighborhood $U$ of $K$. The simplicial map $f\colon X\to Y$ can be realized as the realization of the poset map between barycentric subdivisions. So then one can see that similar to the canonical deformation retraction of $U$ to $K$ that is done simplex by simplex, $X(U)$ also deformation retracts to $X(K)$.
 
 Since the preimage $\eta^{-1}(U)$ is $X(U)\hcoker G$, and the map $X(K)\to X(U)$ is $G$-equivariant, we deduce that the inclusion $\eta^{-1}(K)\hookrightarrow \eta^{-1}(U)$ is a weak homotopy equivalence. We can inductively apply Mayer-Vietoris as follows.
 
 \noindent{\bf Step 3:}  By induction on the dimension of the sub-CW-complex $K$, suppose we know that $\eta^{-1}(K)$ is $\mathcal{L}|_{\eta^{-1}(K)}$-homologically finite for $\text{dim}(K)<k$. The base of the induction is guaranteed by condition (2). Now we consider the sub-CW complex $K\cup_{ [\partial \Delta_\alpha]}  [\Delta_\alpha]$. Let the open neighborhood $U$ of $K$ be as in step $2$. So it is enough to show that $\eta^{-1}(U\cup [\dot{\Delta}_{\alpha}])$ is $\mathcal{L}|_{\eta^{-1}(U\cup [\dot{\Delta}_{\alpha}])}$-homologically finite. Note that this space is covered by open sets $\eta^{-1}(U)$ and $\eta^{-1}( [\dot{\Delta}_{\alpha}])$ whose intersection is the preimage of $ [\dot{\Delta}_{\alpha}]\cap U$. We can choose $U$ so that as in step $2$, the preimage $\eta^{-1}([\dot{\Delta}_{\alpha}]\cap U)$ is homeomorphic to $ ([\dot{\Delta}_{\alpha}]\cap U)\times (g^{-1}(b_{\alpha})\hcoker G)$. So $\eta^{-1}([\dot{\Delta}_{\alpha}]\cap U)$ is also  $\mathcal{L}|_{\eta^{-1}([\dot{\Delta}_{\alpha}]\cap U)}$-homologically finite.
 
 By induction, we know that $\eta^{-1}(U)$ is $\mathcal{L}|_{\eta^{-1}(U)}$-homologically finite. Step $1$ and condition (2)  imply that $\eta^{-1}( [\dot{\Delta}_{\alpha}])$ is of $\mathcal{L}|_{\eta^{-1}( [\dot{\Delta}_{\alpha}])}$-homologically finite. Therefore, Mayer-Vietoris with local coefficient systems (\cite[Chapter VI]{whitehead2012elements}) implies the same for the preimage of $K\cup_{\partial [\Delta_\alpha]}  [\Delta_\alpha]$. Given that $Y/G$ is a finite complex, this process of adding cells ends in a finite step which implies that $X\hcoker G$ is $\mathcal{L}$-homologically finite.

\section{The homological finiteness of $\eta^{-1}(x)$ for a vertex $x$} \label{fiber} Let $x$ be in $\mathcal{F}_0$ in the filtration \ref{filt} which is the image of a separating sphere $S\subset M$.  Our first step to identify the homotopy type of $\eta^{-1}(x)$ is the following proposition. Let $\mathcal{S}(M, [S])$ be the full subcomplex of $\mathcal{S}(M)$  whose vertices are the orbits of $S$ under the action of $\dD(M, \text{rel }\partial)$. 

 \begin{prop}\label{preimage} Let $x$ be in $\mathcal{F}_0$. The preimage $\eta^{-1}(x)$ is homotopy equivalent to 
 \[
 \mathcal{S}(M, [S])\hcoker \dD(M, \text{rel }\partial).
 \]
 \end{prop}
Before proving \Cref{preimage}, let us first observe some properties of the subcomplex  $\mathcal{S}(M, [S])$. 
 Note that in a simplex there could be vertices that are given by isotopic spheres and since they are disjoint, at least for $M$ satisfying the hypothesis of \Cref{Kont}, they bound an embedded $\bS^2\times [0,1]$. By Scharlemann's theorem, we can send two disjoint spheres $S_1$ and $S_2$, by a homeomorphism into $B$. Since the isotopy classes of embedded spheres in $B$ are determined by how an embedded sphere separates the boundary components (see \cite[Proposition 2.2]{hatcher1990finite}), two disjoint isotopic spheres in $B$ bound a $\bS^2\times [0,1]$. Now if the image of $S_1$ and $S_2$ in $B$ are not isotopic in $B$, then they separate disjoint diffeomorphic submanifolds of $M$ in which case, given that each irreducible factor has a nontrivial boundary, the spheres $S_i$ could not be isotopic in $M$ relative to the boundary. So we record this fact as a lemma.
 
\begin{lem}\label{parallel}
Two isotopic essential disjoint separating spheres in $M$ co-bound an embedded $S^2\times [0,1]$.
\end{lem}
\begin{rem}
 Even if $M$ does not have a boundary, the same statement holds. We sketch the argument here in case it might be useful in a more general situation. Let $S$ and $S'$ be two isotopic disjoint spheres in $M$ that are separating. Let $P$ be a submanifold that bounds $S$. Suppose that the isotopy of $M$ (which exists by the isotopy extension theorem) that sends $S$ to $S'$ sends $P$ to the submanifold $P'$. 
Since $M$ might be a closed $3$-manifold, either $P$ is disjoint from $P'$ or one contains the other. For example, if $P$ contains $P'$, then region $P"$ that bounds both $S$ and $S'$ in $P$ should be simply connected. Using the Poincar\' e conjecture and the Schoenflies theorem, we conclude that $P"$ is diffeomorphic to $S^2\times [0,1]$. Now assume that $P$ and $P'$ are disjoint.  So suppose $M$ is the connected sum of $P$ and $P'$ and the region $P"$ that bounds both $S$ and $S'$. If $P"$ is simply connected, then it is diffeomorphic to $S^2\times [0,1]$. So we assume that $P''$ is not simply connected. Since $S$ and $S'$ are essential, $P$ and $P'$ are not simply connected either. Therefore, $\pi_1(M)$ is the free product $\pi_1(P)\ast \pi_1(P")\ast \pi_1(P')$. 

But isotopy acts by the conjugation on the fundamental group and a conjugation cannot send $\pi_1(P)$ factor in the free product to any other factor by using normal forms for elements in a free product. In fact, in our case, there is a simpler argument since all irreducible factors have non-trivial boundaries, we know that $P$ and $P'$ have first homology groups. So the action of the isotopy on the first homology of $M$ would be non-trivial if $P$ and $P'$ were disjoint. 
\end{rem}
\begin{lem}\label{order}
Let $M$ be a compact $3$-manifold as in \Cref{Kont}. Then vertices in each simplex in $\mathcal{S}(M, [S])$ consist of disjoint isotopic spheres and there is a partial order on the vertices of $\mathcal{S}(M, [S])$ 
such that the complex $\mathcal{S}(M, [S])$  is the nerve of this poset structure. 
\end{lem} 
\begin{defn}Let  $\mathcal{S}_{\bullet}(M, [S])$ denote the semisimplicial set given by this ordering of vertices.
\end{defn}
\begin{proof}[Proof of \Cref{order}]
Since the prime decomposition of $M$ has irreducible factors with non-empty boundaries, an edge in $\mathcal{S}(M, [S])$ consists of two disjoint isotopic spheres in the orbit of $S$. This is because if we had two disjoint non-isotopic spheres in the orbit of $S$, given that $S$ is separating, these two spheres cut out homeomorphic pieces $P_1$ and $P_2$ such that neither of them contains the other and they are permuted by an element in $\tH(M, \text{rel }\partial)$. But this is not possible since each prime factor has a nontrivial boundary, the submanifolds $P_1$ and $P_2$ should have nontrivial boundary components. Given that elements in  $\tH(M, \text{rel }\partial)$ fix the boundary pointwise, they cannot permute disjoint submanifolds $P_1$ and $P_2$.  Hence each simplex in $\mathcal{S}(M, [S])$ consists of disjoint isotopic spheres. We call them parallel spheres.

Now to describe the partial order on vertices of $\mathcal{S}(M, [S])$, we need to put an order on parallel spheres. Recall that  $\partial_*M$ denotes the boundary component that contains the base point. Each separating sphere $S$ separates $M$ into connected components and one of them, which we denote by $P_S$, contains the base point. If we have isotopic disjoint separating spheres $S_i$'s, we order them by the inclusion of the components $P_{S_i}$'s. In other words, we can put a metric on $M$ and order $S_i$'s by their distance to the base point. We call this order on spheres of a simplex ``the inside to outside'' order.
\end{proof}

\begin{proof}[Proof of \Cref{preimage}] Let $x\in \mathcal{F}_0$ be an orbit of an isotopy class $[S]$ of a separating sphere $S$.  Let $[\mathcal{S}](M, [S])$ be the full subcomplex of $[\mathcal{S}](M)$ whose vertices are the orbits of $[S]$. Consider the commutative diagram
\begin{equation}\label{q1}
\begin{gathered}
\begin{tikzpicture}[node distance=5cm, auto]
  \node (A) {$\mathcal{S}(M)$};
  \node (B) [right of=A] {$[\mathcal{S}](M)$};
  \node (C) [below of=A, node distance=1.2cm] {$\mathcal{S}(M)/\dD(M, \text{rel }\partial)$};  
  \node (D) [below of=B, node distance=1.2cm] {$[\mathcal{S}](M)/\text{Mod}(M, \text{rel }\partial).$};
 
  \draw[->] (C) to node {$\eta_1$} (D);
  \draw [->] (A) to node {}(C);
  \draw [->] (A) to node {$$} (B);
  \draw [->] (B) to node {$$} (D);
 
\end{tikzpicture}
\end{gathered}
\end{equation}
The preimage of $x$ in $\mathcal{S}(M)$ is the subcomplex $\mathcal{S}(M, [S])$ which is $\dD(M, \text{rel }\partial)$-invariant. So $\eta_1^{-1}(x)$ is the quotient space $\mathcal{S}(M, [S])/\dD(M, \text{rel }\partial)$. Now by the naturality of the Borel construction, we have a  pullback diagram
\begin{equation}\label{q2}
\begin{gathered}
\begin{tikzpicture}[node distance=5cm, auto]
  \node (A) {$\mathcal{S}(M, [S])\hcoker\dD(M, \text{rel }\partial)$};
  \node (B) [right of=A] {$\mathcal{S}(M)\hcoker\dD(M, \text{rel }\partial)$};
  \node (C) [below of=A, node distance=1.2cm] {$\mathcal{S}(M, [S])/\dD(M, \text{rel }\partial)$};  
  \node (D) [below of=B, node distance=1.2cm] {$\mathcal{S}(M)/\dD(M, \text{rel }\partial).$};
 
  \draw[right hook->] (C) to node {$$} (D);
  \draw [->] (A) to node {}(C);
  \draw [right hook->] (A) to node {$$} (B);
  \draw [->] (B) to node {$\eta_2$} (D);
 
\end{tikzpicture}
\end{gathered}
\end{equation}
Hence, the preimage $\eta^{-1}(x)=\eta_2^{-1}(\mathcal{S}(M, [S])/\dD(M, \text{rel }\partial))$ is in fact homeomorphic to the Borel construction $\mathcal{S}(M, [S])\hcoker\dD(M, \text{rel }\partial)$.
\end{proof}
Now we define a semi-simplicial space whose underlying semi-simplicial set is $\mathcal{S}_{\bullet}(M, [S])$.
\begin{defn}\label{top}
$\mathcal{S}^{\tau}_{\bullet}(M, [S])$ is a semi-simplicial space whose $0$-simplices as a set is the same as $\mathcal{S}_{0}(M, [S])$ but it is topologized as the subspace of locally flat embeddings $\text{Emb}^{\sf{lf}}(\bS^2, M)$. By \Cref{order}, the space $\mathcal{S}^{\tau}_{0}(M, [S])$ is a topological poset. The semi-simplicial set $\mathcal{S}^{\tau}_{\bullet}(M, [S])$ is the nerve of this topological poset.
\end{defn}
For a $p$-simplex $e_p\in \mathcal{S}^{\tau}_{p}(M, [S])$, let $\text{Stab}(e_p)$ be the subgroup of  $\tH(M, \text{rel }\partial)$ that fixes $e_p$ pointwise. By \Cref{parallel} and \Cref{order}, each $p$-simplex is given by $p+1$ parallel spheres. Let $\sigma_1=(S_0, S_1, \dots, S_p)$ and $\sigma_2=(S_0', S_1', \dots, S_p')$ be two $p$-simplices where the order of the spheres are induced by the inside to outside order. Since the action of $\dD(M, \text{rel }\partial)$ on the set $\mathcal{S}_{0}(M, [S])$ is transitive, we can find $f_0\in \dD(M, \text{rel }\partial)$ such that $f_0(S_0')=S_0$. Note that $f_0(S_1')$ is isotopic to $S_1$ and they are disjoint from $S_0$. So by the isotopy extension theorem (\cite[Corollary 1.2]{edwards1971deformations}), there exists $f_1\in \dD_0(M, \text{rel }\partial)$ whose support does not intersect $S_0$ and $f_1(S_1')=S_1$. Continuing this process, we can find $f_i\in \dD_0(M, \text{rel }\partial)$ for $1\leq i\leq p$ such that $f_p\circ \dots \circ f_0$ sends $S_i'$ to $S_i$ for all $0\leq i\leq p$. Therefore, the action of $\dD(M, \text{rel }\partial)$  on the set of $p$-simplices is transitive. Hence the topological version of Shapiro's lemma implies that there is a map $\mathrm{B}\text{Stab}^{\delta}(e_p)\to \mathcal{S}_{p}(M, [S])\hcoker \dD(M, \text{rel }\partial)$ that is a weak equivalence. To show that the homotopy quotient $\mathcal{S}^{\tau}_{p}(M, [S])\hcoker \tH(M, \text{rel }\partial)$ is homotopy equivalent to $\mathrm{B}\text{Stab}(e_p)$, we need the following lemma.
Let $\text{Sing}_{\bullet}(X)$ denote the singular set of a topological space $X$. Recall that by \cite{MR0084138} and \cite[Lemma 1.7]{MR3995026}, the augmentation map $ \text{Sing}_{\bullet}(X)\to X$ induces a weak equivalence after fat realization.
\begin{lem}\label{stab}
There is a map $\mathrm{B}|\textnormal{Sing}_{\bullet}(\text{Stab}(e_p))|\to\mathcal{S}^{\tau}_{p}(M, [S])\hcoker \tH(M, \text{rel }\partial)$ that is a weak equivalence.
\end{lem}
\begin{proof}
Let  $\mathcal{S}^{\tau}_{p}(M, [S])_{\bullet}$ be defined similarly to $\mathcal{S}^{\tau}_{p}(M, [S])$ in \Cref{top} but instead of the space of locally flat embeddings, we consider the simplicial set of locally flat embeddings $\text{Emb}^{\sf{lf}}(\bS^2, M)_{\bullet}$ (see \cite[Definition 2.5]{nariman2020local}).
Then the action of $\tH(M, \text{rel }\partial)$ on $\mathcal{S}^{\tau}_{p}(M, [S])$ induces the following map of simplicial sets
\[
\text{Sing}_{\bullet}(\tH(M, \text{rel }\partial))\to \mathcal{S}^{\tau}_{p}(M, [S])_{\bullet},
\]
which by the parametrized isotopy extension theorem, it is a Kan fibration (see \cite[Page 19]{MR0356103} where it is called {\it A.I.T which is short for the ambient isotopy extension theorem}) whose fiber is $\text{Sing}_{\bullet}(\text{Stab}(e_p))$. 

Again the augmentation map $\mathcal{S}^{\tau}_{p}(M, [S])_{\bullet}\to \mathcal{S}^{\tau}_{p}(M, [S])$ induces a weak equivalence after fat realization and it is equivariant with respect to the map $\text{Sing}_{\bullet}(\tH(M, \text{rel }\partial))\to \tH(M, \text{rel }\partial)$. So we obtain a natural weak equivalence
\[
||\mathcal{S}^{\tau}_{p}(M, [S])_{\bullet}\hcoker \text{Sing}_{\bullet}(\tH(M, \text{rel }\partial))||\xrightarrow{\simeq}\mathcal{S}^{\tau}_{p}(M, [S])\hcoker \tH(M, \text{rel }\partial).
\]
Therefore, it is enough to show that for each $k$ in the simplicial direction, we have a map
\[
\mathrm{B}\textnormal{Sing}_{k}(\text{Stab}(e_p))\to \mathcal{S}^{\tau}_{p}(M, [S])_{k}\hcoker \text{Sing}_{k}(\tH(M, \text{rel }\partial)),
\]
that induces a weak equivalence.

Let $G$ be the discrete group $\text{Sing}_{k}(\tH(M, \text{rel }\partial))$ and let $H$ be its subgroup $\textnormal{Sing}_{k}(\text{Stab}(e_p))$. Note that there is natural isomorphism from the coset $G/H$ to $\mathcal{S}^{\tau}_{p}(M, [S])_{k}$. As is also mentioned in \cite[Section 3.1.1]{nariman2020local}, the Shapiro's lemma for discrete groups implies that there is a natural map $\mathrm{B}H\to (G/H)\hcoker G$ which is a weak homotopy equivalence.
\end{proof}

Now note that Thurston's homology isomorphism \ref{Thurston}, in the generality that McDuff proved in \cite[Section 2]{mcduff1980homology} implies that the map
\[
\mathrm{B}\text{Stab}^{\delta}(e_p)\to \mathrm{B}\text{Stab}(e_p),
\]
is an acyclic map. This map factors through $\mathrm{B}\text{Stab}^{\delta}(e_p)\to\mathrm{B}|\textnormal{Sing}_{\bullet}(\text{Stab}(e_p))|$ which is induced by mapping $\text{Stab}^{\delta}(e_p)$ to the $0$-simplices $\textnormal{Sing}_{0}(\text{Stab}(e_p))$. So using \Cref{stab}, we obtain that for each $p$, the natural map
\begin{equation}\label{isostab}
f_p\colon\mathcal{S}_{p}(M, [S])\hcoker \dD(M, \text{rel }\partial)\to \mathcal{S}^{\tau}_{p}(M, [S])\hcoker \tH(M, \text{rel }\partial),
\end{equation}
induces an acyclic map. We claim that the induced map between the realizations
\[
f\colon ||\mathcal{S}_{\bullet}(M, [S])\hcoker \dD(M, \text{rel }\partial)||\to ||\mathcal{S}^{\tau}_{\bullet}(M, [S])\hcoker \tH(M, \text{rel }\partial)||
\]
is also acyclic. Because for a local coefficient $\mathcal{L}$ on the realization $||X_{\bullet}||$ of a semisimplicial space $X_{\bullet}$,
there is a  spectral sequence that calculates the homology of the realization with local coefficients (see \cite[Section 1.4]{MR3995026})
\[
E^1_{p,q}\cong H_p( X_q; \mathcal{L}_q)\Rightarrow H_{p+q}( ||X_{\bullet}||; \mathcal{L}),
\]
where $\mathcal{L}_q$ is the local coefficient on the space of $q$-simplices by pulling back $\mathcal{L}$ via the map
\[
X_q\times b_q \to X_q\times \Delta^q\to ||X_{\bullet}||,
\]
where $b_q$ is the barycenter of $\Delta^q$. The maps $f_{\bullet}$ induce isomorphisms on $E^1$-pages of the corresponding spectral sequences. Therefore, given a local coefficient $\mathcal{L}$ on $||\mathcal{S}^{\tau}_{\bullet}(M, [S])\hcoker \tH(M, \text{rel }\partial)||$, the map $f$ induces a homology isomorphism with local coefficients on $||\mathcal{S}_{\bullet}(M, [S])\hcoker \dD(M, \text{rel }\partial)||$ that is the pull back $f^*(\mathcal{L})$. Hence, to prove weak homological finiteness for $\eta^{-1}(x)$, we prove strong homological finiteness of $$||\mathcal{S}^{\tau}_{\bullet}(M, [S])\hcoker \tH(M, \text{rel }\partial)||,$$ by finding a model for it that sits in a homotopy fiber sequence \ref{fib1} to be able to argue inductively on the number of prime factors.

 We can define the smooth version of $\mathcal{S}^{\tau}_{\bullet}(M, [S])$ and work with diffeomorphism groups. But in this dimension and for codimension $1$ embeddings, the corresponding objects in the $C^0$ and $C^{\infty}$-category are weakly homotopy equivalent. So we stick to the $C^0$-category. 
 
 \section{Parallel spheres and bar constructions}\label{barconstruction} Let $e\colon\partial_* M\hookrightarrow \{0\}\times \bR^{\infty}$ be a fixed locally flat embedding of the boundary component that contains the base point and let $\text{Emb}^{\sf{lf}}_{\partial}(M, [0,\infty)\times \bR^{\infty})$ be the space of locally flat embeddings of $M$ whose intersection with $\{0\}\times \bR^{\infty}$ is  $e(\partial_* M)$. Lashof in \cite[Appendix, theorem 1]{lashof1976} considered three variants of spaces of topological embeddings that are Kan complexes and he showed that they  are homotopy equivalent as long as the dimension of the target is larger than $4$ and the codimension of the embedding is at least $3$. One of these variants is the singular set of the on the space of locally flat embeddings. The proof in \cite[Lemma 2.2]{kupers2015proving} implies that one of Lashof's model for $\text{Emb}^{\sf{lf}}_{\partial}(M, [0,\infty)\times \bR^{\infty})$ is weakly contractible.  Therefore, the space $\text{Emb}^{\sf{lf}}_{\partial}(M, [0,\infty)\times \bR^{\infty})$ that is homotopy equivalent to the realization of its singular set (\cite{MR0084138}), is also weakly contractible. Working simplicially first as in \cite[Lemma 2.6 and Lemma 2.7]{kupers2015proving}, and then geometrically realizing, we deduce that $\text{Emb}^{\sf{lf}}_{\partial}(M, [0,\infty)\times \bR^{\infty})/\tH(M, \text{rel }\partial)$ is a model for the classifying space $\mathrm{B}\tH(M, \text{rel }\partial)$ and the semi-simplicial space
 \[
 \mathcal{M}_{\bullet}(M, [S])\coloneqq \frac{\mathcal{S}^{\tau}_{\bullet}(M, [S])\times \text{Emb}^{\sf{lf}}_{\partial}(M, [0,\infty)\times \bR^{\infty})}{\tH(M, \text{rel }\partial)},
 \]
is level-wise weakly equivalent to $\mathcal{S}^{\tau}_{\bullet}(M, [S])\hcoker \tH(M, \text{rel }\partial)$. We think of $ \mathcal{M}_{\bullet}(M, [S])$ as a configuration space of the manifolds in $[0,\infty)\times \bR^{\infty}$ that are homeomorphic to $M$ satisfying the boundary condition and with a choice of parallel spheres in the orbit of $S$. 

Now we shall define a two-sided bar construction model for $ \mathcal{M}_{\bullet}(M, [S])$. Let $\iota_0\colon \bS^2\hookrightarrow\{0\}\times \bR^{\infty}$ be a fixed embedding and we denote the embedding $\iota_0+t\cdot e_1$ in $\{t\}\times \bR^{\infty}$ by $\iota_t$.
\begin{defn} We define the space $\mathcal{M}(M)$ to be the quotient space $\text{Emb}^{\sf{lf}}_{\partial}(M, [0,\infty)\times \bR^{\infty})/{\tH(M, \text{rel }\partial)}$. This is a model for the classifying space $\mathrm{B}\text{Homeo}(M, \text{rel }\partial)$. For an element $f\in \mathcal{M}(M)$, we let $\text{image}(f)$  be the unparametrized submanifold of $[0,\infty)\times \bR^{\infty}$  given by $f$ that is homeomorphic to $M$.
\end{defn}
\begin{defn}\label{cylindar}
Let $\mathfrak{D}$ be the topological monoid given by space of pairs $(t, f)\in [0,\infty)\times \mathcal{M}(\bS^2\times [0,1])$ where 
\begin{itemize}
\item we have $\text{image}(f)\subset [0,t]\times \bR^{\infty}$ and
\item  the intersection of $\text{image}(f)$ with $\{0\}\times \bR^{\infty}$ and $\{t\}\times \bR^{\infty}$ are given by embeddings $\iota_0$ and $\iota_t$ respectively. 
\end{itemize}
The monoid structure is given by adding the $t$-coordinates and stacking the embeddings next to each other. 
\end{defn}
 It is standard to see that the topological monoid $\mathfrak{D}$ is homotopy equivalent to $\BD(\bS^2\times [0,1], \text{rel }\partial)$. The homotopy type of $\tH(\bS^2\times [0,1], \text{rel }\partial)$ is known (\cite[Appendix]{hatcher1983proof}) to be the loop space $\Omega(\text{SO}(3))$. Also recall that the inclusion $\text{SO}(3)\hookrightarrow \tH_0(\bS^2)$ is a weak equivalence (\cite[Theorem 1.2.2]{hamstrom1974homotopy}). 
 Recall that when we cut $M$ along $S$, we obtain two pieces $M_1$ and $M_2$ where $M_1$ contains $\partial_* M$, the boundary component of $M$ with the base point. Now we define moduli space models for $\BD(M_1, \text{rel }\partial)$ and $\BD(M_2, \text{rel }\partial)$ that are modules over the topological monoid $\mathfrak{D}$.
 
 \begin{defn}\label{module1}
 Let $\mathcal{L}$ be the space of pairs $(t, f)\in [0,\infty)\times \mathcal{M}(M_1)$ such that 
 \begin{itemize}
 \item The image($f$) lies in the strip $[0,t]\times \bR^{\infty}$.
 \item  The intersection $\text{image}(f)\cap \{0\}\times \bR^{\infty}$ is given by the embedding $e$ (the embedding of the base boundary component) and $\text{image}(f)\cap \{t\}\times \bR^{\infty}$ is given by $\iota_t$.
 \end{itemize}
 Similarly, let $\mathcal{R}$ to be the space of pairs $(t, f)\in [0,\infty)\times \mathcal{M}(M_2)$ such that 
 \begin{itemize}
 \item The image $\text{image}(f)$ lies in $[t,\infty)\times \bR^{\infty}$.
 \item  The intersection  $\text{image}(f)\cap \{t\}\times \bR^{\infty}$ is given by $\iota_t$.
 \end{itemize}
 \end{defn}
 
It is easy to see that  $\mathcal{L}$ and $\mathcal{R}$ are weakly equivalent to $\BD(M_1, \text{rel }\partial)$ and $\BD(M_2, \text{rel }\partial)$ respectively.

 Note that there is a right $\mathfrak{D}$-module structure on $\mathcal{L}$ such that the action of $(t,f)\in \mathfrak{D}$ on $(t', f')\in \mathcal{L}$ is the pair $(t+t', f' \sqcup (f+t'\cdot e_1))$ where $f+t'\cdot e_1$ is the embedding $f$ shifted in the first coordinate to the right by $t'$. Similarly, there is a left $\mathfrak{D}$-module structure on $\mathcal{R}$.
 
  \begin{figure}[h]

\begin{tikzpicture}[scale=.4]
\begin{scope}[shift={(7,0)}]
\draw[line width=1.05pt] (.46,-3)--(-.8,- 6.7);
\draw[line width=1.05pt] (.46,5.6)--(-.8,1.3);
\draw[line width=1.05pt] (-.8,- 6.7)--(-.8,1.3);
\draw[line width=1.05pt] (.46,-3)--(.46,-2.5);
\draw[line width=1.05pt] (.46,5.6)--(.35,2.5);
\begin{scope}[shift={(9,0)}]
\draw[line width=1.05pt] (.46,-3)--(-.8,- 6.7);
\draw[line width=1.05pt] (.46,5.6)--(-.8,1.3);
\draw[line width=1.05pt] (-.8,- 6.7)--(-.8,1.3);
\draw[line width=1.05pt] (.46,-3)--(.46,-2.5);
\draw[line width=1.05pt] (.46,5.6)--(.35,2.5);
\end{scope}
\begin{scope}[shift={(-12,0)}]
\draw[line width=1.05pt] (.46,-3)--(-.8,- 6.7);
\draw[line width=1.05pt] (.46,5.6)--(-.8,1.3);
\draw[line width=1.05pt] (-.8,- 6.7)--(-.8,1.3);
\draw[line width=1.05pt] (.46,-3)--(.46,-2.5);
\draw[line width=1.05pt] (.46,5.6)--(.35,2.5);
\end{scope}
\draw[line width=1.05pt] [dashed] (0,.-2.5) arc (-90:90:0.5 and 2.5);
\draw[line width=1.05pt] [dashed] (9,-2.5) arc (-90:90:0.5 and 2.5);
\draw[line width=1.05pt]  (9,2.5) arc (90:270:0.5 and 2.5);
\draw[line width=1.05pt] [dashed] (.46,.-2.5)--(0,-2.5);
\draw[line width=1.05pt]  (-.8,.-2.5)--(-5,-2.5);
\draw[line width=1.05pt] [dotted] (-.8,.-2.5)--(0,-2.5);

\draw[line width=1.05pt]  [dashed] (-5,-2.5)--(-7,-2.5);
\draw[line width=1.05pt] [dashed] (.46,2.5)--(0,2.5);
\draw[line width=1.05pt]  (-.5,2.5)--(-5,2.5);

\draw[line width=1.05pt] [dashed] (-12,-2.5) arc (-90:90:0.5 and 2.5);
\draw[line width=1.05pt]  (-12,2.5) arc (90:270:0.5 and 2.5);
\draw [line width=1.05pt] (-12,2.5)--(-3,2.5);
\draw [line width=1.05pt] (-12,-2.5)--(-3,-2.5);

\draw[line width=1.05pt]  [dashed] (-5,2.5)--(-7,2.5);
\draw [line width=1.05pt] (0,2.5) arc (90:270:0.5 and 2.5);
\draw [line width=1.05pt] (0,2.5)--(8.5,2.5);
\draw [line width=1.05pt] [dotted] (8.5,2.5)--(9,2.5);

\draw [line width=1.05pt] (9,2.5)--(16.4,2.5);

\draw [line width=1.05pt] [dashed](16.4,2.5)--(18.2,2.5);
\draw [line width=1.05pt] (0,-2.5)--(8,-2.5);
\draw [line width=1.05pt] [dotted] (8,-2.5)--(9,-2.5);

\draw [line width=1.05pt] (9,-2.5)--(16.4,-2.5);
%
\draw [line width=1.05pt] [dashed](16.4,-2.5)--(18.2,-2.5);
\begin{scope}[shift={(-9,0)}]
\draw [line width=1.05pt] (4.6,.5) arc (240:300:2.5 and 6.75);
\draw   [line width=1.7pt] (6.87, 0.13) arc (48:150:1.2 and 0.9);
\end{scope}
\begin{scope}[shift={(-14,0)}]
\draw [line width=1.05pt] (4.6,.5) arc (240:300:2.5 and 6.75);
\draw   [line width=1.7pt] (6.87, 0.13) arc (48:150:1.2 and 0.9);
\end{scope}
\begin{scope}[shift={(9,0)}]
\draw [line width=1.05pt] (4.6,.5) arc (240:300:2.5 and 6.75);
\draw   [line width=1.7pt] (6.87, 0.13) arc (48:150:1.2 and 0.9);
\end{scope}
\node at (0,-8) {$\{a\}\times \bR^{\infty}$};
\node at (9,-8) {$\{b\}\times \bR^{\infty}$};
\node at (-12,-8) {$\{0\}\times \bR^{\infty}$};
\node at (4.5,-6) {$[0,1]\times \bS^2$};
\node at (-6.5,-6) {$M_1$};
\node at (-6.5,-4) {$\rotcong$};
\node at (4.5,-4) {$\rotcong$};
\node at (14,-4) {$\rotcong$};
\node at (14,-6) {$M_2$};

\end{scope}
\end{tikzpicture} 
\caption{Schematic picture in one dimension lower on how $\mathrm{BD}$ acts on $\mathrm{BR}$ and $\mathrm{BL}$}
\end{figure}
We consider the two-sided bar resolution given by the semi-simplicial space
\[
B_{p}(\mathcal{L}, \mathfrak{D}, \mathcal{R})= \mathcal{L}\times \mathfrak{D}^p\times \mathcal{R}
\]
where the face map $d_0$ and $d_p$ are given by the actions of $\mathfrak{D}$ on $\mathcal{L}$ and $\mathcal{R}$ respectively and other face maps are induced by the monoid structure of $\mathfrak{D}$. 

Note that there is a natural semi-simplicial map 
\[
h_p\colon B_{p}(\mathcal{L}, \mathfrak{D}, \mathcal{R})\to \mathcal{M}_{p}(M, [S])
\]
by gluing the embeddings in order and choosing the spheres along which the embeddings are glued as a choice of parallel spheres in the orbit of $S$. On the other hand, recall that the action of $\tH(M,\text{rel }\partial)$ on $\mathcal{S}^{\tau}_{p}(M, [S])$  is transitive for each $p$. For a $p$-simplex $\sigma \in \mathcal{S}^{\tau}_{p}(M, [S])$, the homotopy quotient $\mathcal{S}^{\tau}_{p}(M, [S])\hcoker \tH(M,\text{rel }\partial)$ is weakly equivalent to $\mathrm{B}\text{Stab}(\sigma)$ by \Cref{stab}, and the space $\mathcal{M}_{p}(M, [S])$ is also weakly equivalent to $\mathrm{B}\text{Stab}(\sigma)$. As we shall see below  the semi-simplicial map $h_{\bullet}$ is level-wise a weak equivalence, and we have weak equivalences between the  (fat) realizations
\begin{equation}\label{bar}
||B_{\bullet}(\mathcal{L}, \mathfrak{D}, \mathcal{R})||\xrightarrow{\simeq} ||\mathcal{M}_{\bullet}(M, [S])||\simeq ||\mathcal{S}^{\tau}_{\bullet}(M, [S])\hcoker \tH(M,\text{rel }\partial)||.
\end{equation}
\begin{lem}
    The map $h_p$ is a weak equivalence for each integer $p>0$.
\end{lem}
\begin{proof}
    Suppose $M$ is embedded in $\bR\times \bR^{\infty}$ such that the spheres in the simplex $\sigma$ are embedded in slices $\{a_i\}\times \bR^\infty$ for some real numbers $a_i$. Then $\text{Stab}(\sigma)$ is isomorphic to the product of homeomorphism groups of submanifolds between slices. So we can use these slices to obtain a natural map $\mathrm{B}\text{Stab}(\sigma)$ to $B_{p}(\mathcal{L}, \mathfrak{D}, \mathcal{R})$ which is a weak equivalence. Therefore, it is enough to show that the composition
    \[
   \tilde{h}_p\colon \mathrm{B}\text{Stab}(\sigma)\to \mathcal{M}_{p}(M, [S]),
    \]
    is a weak equivalence. It is easier to show this in the $C^\infty$-category. Remember in this dimension, the corresponding objects of our interests are weakly equivalent. We denote the smooth versions by superscript $\infty$. We have a weak equivalence $\iota\colon\mathrm{B}\text{Stab}^\infty(\sigma)\to \mathrm{B}\text{Stab}(\sigma)$. Note that $ \tilde{h}^\infty_p\coloneqq\tilde{h}_p\circ \iota$ lands in 
    \[
    \mathcal{M}^\infty_{p}(M, [S])\coloneqq \frac{\mathcal{S}^{\infty}_{\bullet}(M, [S])\times \text{Emb}^{\infty}_{\partial}(M, [0,\infty)\times \bR^{\infty})}{\Diff(M, \text{rel }\partial)}.
    \]
    Given that corresponding objects in $C^\infty$ and $C^0$ are weakly equivalent here, it is enough to show that $\tilde{h}^\infty_p$ is a weak equivalence. 

    Let $G=\Diff(M, \text{rel }\partial)$ and $H$ be the subgroup $\text{Stab}^\infty(\sigma)$. By the parametrized isotopy extension theorem $H$ has a local section and $H\to G\to G/H$ is a principal $H$-bundle. Therefore, similar to the proof of \Cref{stab}, we have a topological Shapiro's lemma where in this case the map  $\tilde{h}^\infty_p\colon\mathrm{B}H\xrightarrow{\simeq} (G/H)\hcoker G$ is a weak equivalence.
\end{proof}
Since the topological monoid $ \mathfrak{D}$ is path-connected, similar to \cite[Theorem 4.5]{kupers2019some}, we have a homotopy fiber sequence
\begin{equation}\label{fib4}
 \mathcal{R}\to ||B_{\bullet}(\mathcal{L}, \mathfrak{D}, \mathcal{R})||\to ||B_{\bullet}(\mathcal{L}, \mathfrak{D},*)||.
\end{equation}

We shall use the technique of the Weiss fibration as was explained in \cite{kupers2019some} to show that this is the desired homotopy fiber sequence \ref{fib1}.
\subsection{Kupers' bar resolution for self-embeddings} 
We shall use Kupers' theorem (\cite[Section 4]{kupers2019some}) to determine the homotopy type of $||B_{\bullet}(\mathcal{L}, \mathfrak{D},*)||$ to be able to say that is weakly homologically finite.

The manifold $M_1$ has a sphere boundary component $\partial_0 M_1=S$ which we call the free boundary component and we denote the union of the rest of the boundary components by $\partial_1M_1$ which we call the fixed boundary components. Let $\tH(M_1, S, \text{rel }\partial_1)$ be the group of homeomorphisms of $M_1$ that fix $S$ set wise and fix a neighborhood of $\partial_1M_1$ pointwise.

\begin{thm}\label{Kupers'}
There is a zig-zag of weak equivalences between the bar resolution $||B_{\bullet}(\mathcal{L}, \mathfrak{D},*)||$ and the classifying space $\BD(M_1, S, \text{rel }\partial_1)$. 
\end{thm}
Kupers in \cite{kupers2019some} gives a model for {\it Weiss fiber sequence} where the set-up is that we have an $n$-dimensional manifold $M$ with a non-empty boundary and we fix an embedded $\bD^{n-1}\hookrightarrow \partial M$. Let $\text{Emb}^{\cong}_{1/2 \partial}(M)$ be the space of self-embeddings of $M$ that are identity on $\partial M\backslash \text{int}(\bD^{n-1})$ and are isotopic to a diffeomorphism that fixes the boundary through isotopies fixing $\partial M\backslash \text{int}(\bD^{n-1})$. There exists a fiber sequence named after Michael Weiss
\[
\BDiff(M, \text{rel }\partial)\to \mathrm{B}\text{Emb}^{\cong}_{1/2 \partial}(M)\to \mathrm{B}\BDiff(\bD^n, \text{rel }\partial), 
\]
where the delooping $ \mathrm{B}\BDiff(\bD^n, \text{rel }\partial)$ is defined by considering $\BDiff(\bD^n, \text{rel }\partial)$ as a topological monoid similar to \Cref{cylindar} and the $E_1$-structure on this topological monoid is given by stacking along the first coordinate when we consider the interior of the cube as a model for the interior of the disk. We want to use a similar fiber sequence for a compact $3$-manifold $M$ with a non-empty boundary where at least one of its boundary components is homeomorphic to $\bS^2$.
\begin{proof}[Proof of \Cref{Kupers'}]  Let $\text{Emb}^{\cong}_{\partial_1}(M)$ the space of self locally flat embeddings of $M_1$ that are the identity on the fixed boundary components $\partial_1 M_1$ and are isotopic to a homeomorphism that fixes the boundary through isotopies fixing $\partial_1 M_1$. 

Given that in dimension $3$, the corresponding objects in $C^0$ and $C^{\infty}$ category are weakly equivalent (\Cref{CH}), we may apply the proof of \cite[Theorem 4.17]{kupers2019some} mutatis mutandis to conclude that there is a fiber sequence
\begin{equation}\label{fibb3}
\BD(\bS^2\times [0,1], \text{rel }\partial)\to \BD(M_1, \text{rel }\partial)\to \mathrm{B}\text{Emb}^{\cong}_{\partial_1}(M_1)
\end{equation}
that is induced by the natural inclusions $$\tH(\bS^2\times [0,1], \text{rel }\partial)\to \tH(M_1, \text{rel }\partial)\to \text{Emb}^{\cong}_{\partial_1}(M_1).$$ 
Moreover, his method shows that there is a weak equivalence $||B_{\bullet}(\mathcal{L}, \mathfrak{D},*)||\simeq \mathrm{B}\text{Emb}^{\cong}_{\partial_1}(M_1)$.  On the other hand, we have a fiber sequence
\[
\tH(M_1, \text{rel }\partial)\to \tH(M_1, S, \text{rel }\partial_1)\to \tH_0(\bS^2),
\]
where the last map is the restriction to $S$. Since homeomorphisms in $\tH(M_1, S, \text{rel }\partial_1)$ fix at least one boundary component, they are orientation preserving so they restrict to $\tH_0(\bS^2)$. This is a $\tH(M_1, \text{rel }\partial)$-bundle over $\tH_0(\bS^2)$. So we obtain a fiber sequence
\begin{equation}\label{fib3}
\tH_0(\bS^2)\to \BD(M_1, \text{rel }\partial)\to \BD(M_1, S, \text{rel }\partial_1),
\end{equation}
where the first map is the classifying map for the $\tH(M_1, \text{rel }\partial)$-bundle over $\tH_0(\bS^2)$.

Note that the group $\tH(M_1, S, \text{rel }\partial_1)$ is a submonoid of $\text{Emb}^{\cong}_{\partial_1}(M_1)$ so we have a natural map between their classifying spaces induced by this inclusion. The advantage of using $\mathrm{B}\text{Emb}^{\cong}_{\partial_1}(M_1)$ over the bar construction $||B_{\bullet}(\mathcal{L}, \mathfrak{D},*)||$ is that we have an explicit map 
\[
\BD(M_1, S, \text{rel }\partial_1)\to\mathrm{B}\text{Emb}^{\cong}_{\partial_1}(M_1)
\]
that we shall argue that is a weak equivalence.

To show that the base spaces of the two fiber sequences (\ref{fibb3} and \ref{fib3}) are weakly equivalent, we need maps of fiber sequences that induce weak equivalences between fibers and total spaces. They have the same total spaces so let us describe the map between fibers.

Let $\tH(\bS^2\times[0,1], S, \text{rel }\partial)$ be the subgroup of $\tH(\bS^2\times [0,1])$ that fixes $\bS^2\times\{1\}$ pointwise and fixes $\bS^2\times\{0\}$ set-wise. We have a $\tH(\bS^2\times[0,1],\text{rel }\partial)$-fiber sequence
\begin{equation}\label{S^2}
\tH(\bS^2\times[0,1],\text{rel }\partial)\to \tH(\bS^2\times[0,1], S, \text{rel }\partial)\to \tH_0(\bS^2),
\end{equation}
where the last map is the restriction to $\bS^2\times\{0\}$. This fiber sequence is classified by a map
\[
f\colon\tH_0(\bS^2)\to \BD(\bS^2\times[0,1],\text{rel }\partial).
\]
The total space of the fiber sequence (\ref{S^2}) is contractible by Hatcher (\cite[Appendix]{hatcher1983proof}). So the classifying map $f$ is a weak equivalence. Therefore, we obtain a map of fiber sequences 


\begin{equation}\label{commdiagm1}
\begin{gathered}
\begin{tikzpicture}[node distance=5cm, auto]
  \node (A) {$\BD(M_1, \text{rel }\partial)$};
  \node (B) [right of=A] {$\mathrm{B}\tH(M_1, \text{rel }\partial)$};
  \node (C) [below of=A, node distance=1.2cm] {$\mathrm{B}\text{Emb}^{\cong}_{\partial_1}(M_1)$};  
  \node (D) [below of=B, node distance=1.2cm] {$\BD(M_1, S, \text{rel }\partial_1)$};
 \node (A') [above of=A, node distance=1.2cm] {$\BD(\bS^2\times [0,1], \text{rel }\partial)$};
 \node (B') [above of=B, node distance=1.2cm] {$\tH_0(\bS^2)$};

  \draw[<-] (C) to node {$$} (D);
  \draw [->] (A) to node {}(C);
  \draw [<-] (A) to node {$=$} (B);
  \draw [->] (B) to node {$$} (D);
                   \draw [->] (A') to node {$$} (A);
                 \draw [->] (B') to node {$$} (B);
                  \draw [<-] (A') to node {$\simeq$} (B');
 
\end{tikzpicture}
\end{gathered}
\end{equation}

It is easy to see that the diagram \ref{commdiagm1} commutes. Given the indicated weak equivalences in the above diagrams, the comparison of the long exact sequence of the homotopy groups for the fiber sequences implies that the bottom map is a weak equivalences between $ \BD(M_1, S, \text{rel }\partial_1)$ and $\mathrm{B}\text{Emb}^{\cong}_{\partial_1}(M_1)$.
\end{proof}
By induction, \Cref{Kont} implies that $\BD(M_1, S, \text{rel }\partial_1)$ and $\mathcal{R}$ are strongly homologically finite. Therefore, \Cref{kupersfinitenesslemma} implies that in the homotopy fiber sequence (\ref{fib4}), the total space is strongly homologically finite which in turn implies the same for $\eta^{-1}(x)$ when $x\in \mathcal{F}_0$.

For the inductive argument to analyze $\eta^{-1}(x)$ when $x$ is in higher filtration, we need a slight generalization of \Cref{Kupers'} whose proof is the same. Let $M_1$ be a compact $ 3$ manifold such that we decompose the boundary components into three subsets: $\partial_{-1}=S_{\text{free}}$ which is a subset of spherical boundary components, $\partial_0=S$ which is a spherical boundary component that is not in $\partial_{-1}$ and $\partial_1$ which is the union of the rest of the boundary components. Let $\tH(M_1, S_{\text{free}}, \text{rel }\partial_1\cup \partial_0)$ be the subgroup of $\tH(M_1, \text{rel }\partial_1\cup \partial_0)$ which fixes each boundary component in $\partial_{-1}$ set-wise and let $\tH(M_1, S_{\text{free}}\cup S, \text{rel }\partial_1)$ be the subgroup of $\tH(M_1, \text{rel }\partial_1)$ which fixes each  boundary component in $S_{\text{free}}\cup S$ set-wise. Similarly, one can define a $\mathfrak{D}$-module that has the homotopy type of the classifying space $\BD(M_1, S_{\text{free}}, \text{rel }\partial_1\cup \partial_0)$ where module structure is induced by gluing to the boundary component $\partial_0=S$.
\begin{thm}\label{Kupers}
There is a zig-zag of weak equivalences between the bar resolution $||B_{\bullet}(\mathcal{L}, \mathfrak{D},*)||$ and the classifying space $\BD(M_1, S_{\text{free}}\cup S, \text{rel }\partial_1)$. 
\end{thm}
So roughly speaking,   the construction $||B_{\bullet}(\mathcal{L}, \mathfrak{D},*)||$ makes one more spherical boundary component (the component $\partial_0$) free.
\begin{rem}\label{rem1}
 For $M$ being connected sum of {\it two} irreducible $3$-manifolds with non-empty boundaries, Hatcher's theorem (\cite{hatcher1981diffeomorphism}) about $3$-manifolds also implies Kontsevich's finiteness.   Let $S$ be a separating sphere in $M$, then his theorem implies that $\iota\colon\BDiff(M, S, \text{rel }\partial)\to \BDiff(M,\text{rel }\partial)$ is a homotopy equivalence where $\Diff(M, S, \text{rel }\partial)$ is the subgroup of $\Diff(M,\text{rel }\partial)$ that fixes $S$ set wise. We have a homotopy fiber sequence
\[
\BDiff(M_2,\text{rel }\partial)\to \BDiff(M, S,\text{rel }\partial)\to \BDiff(M_1, S, \text{rel }\partial M),
\]
where $M_1$ and $M_2$ are obtained by cutting $M$ along $S$. This homotopy fiber sequence is similar to the one in (\ref{fib1}).
\end{rem}
\section{Higher filtrations and the last step of the proof of \Cref{main'}} \label{higher}For $k>0$  suppose $x\in \mathcal{F}_k-\mathcal{F}_{k-1}$. We want to generalize the above bar resolution model by iterating the same construction $(k+1)$ times for each separating sphere in different orbits. Then we write this iterated bar construction in a fiber sequence whose base and fiber, by induction on the number of prime factors and by applying \Cref{R} on free sphere boundary components, are weakly homologically finite.

 Let  ${\bf S}=\{S_0, S_1,\dots, S_k\}$ be a set of $k+1$  separating spheres where $S_i$'s are pairwise in different orbit classes under the action of $\dD(M, \text{rel }\partial)$. We pick an order on these spheres and note that they cannot be permuted via the action of  $\dD(M, \text{rel }\partial)$. Let $\Delta([{\bf S}])$ be a $k$-simplex in $[S](M)$ whose vertices are the isotopy classes $[{\bf S}]=\{[S_0], [S_1], \dots, [S_k]\}$. As in \Cref{last}, let $[\Delta([{\bf S}])]$ be the cell in $[S](M)/\text{Mod}(M, \text{rel }\partial)$ that is the image of the simplex $\Delta([{\bf S}])$ and let $[\dot{\Delta}([{\bf S}])]$ be the interior of this cell. Suppose the point $x$ lies in $[\dot{\Delta}([{\bf S}])]$. 

 \[
\begin{tikzpicture}[node distance=5cm, auto]
  \node (A) {$\mathcal{S}(M)$};
  \node (B) [right of=A] {$[\mathcal{S}](M)$};
  \node (C) [below of=A, node distance=1.2cm] {$\mathcal{S}(M)\hcoker\dD(M, \text{rel }\partial)$};  
  \node (D) [below of=B, node distance=1.2cm] {$[\mathcal{S}](M)/\text{Mod}(M, \text{rel }\partial)$};
 
  \draw[->] (C) to node {$\eta$} (D);
  \draw [->] (A) to node {}(C);
  \draw [->] (A) to node {$q$} (B);
  \draw [->] (B) to node {$\pi$} (D);
 
\end{tikzpicture}
 \]
As we mentioned in step $0$ of \Cref{mainlem}, the restriction of $p$ to the preimage of $\dot{\Delta}([{\bf S}])$ is a trivial fiber bundle. Let us recall part of the statement and the proof of \cite[Theorem 1.3.1]{MR0184242} that is helpful to determine the preimage $q^{-1}(y)$ for a point $y\in \dot{\Delta}([{\bf S}])$.

If $f\colon X\to \Delta$ is a simplicial map from a simplicial complex $X$ to a simplex $\Delta$, then for each $y\in \dot{\Delta}$, there is a homeomorphism $h\colon f^{-1}(y)\times \dot{\Delta}\to f^{-1}(\dot{\Delta})$. We can choose $h$ so that for a subcomplex $L$ of $X$, we have 
\begin{equation}\label{*}
h^{-1}(L)=(f^{-1}(y)\cap L)\times \dot{\Delta}.
\end{equation}
As it is explained in \cite[Lemma 1.7 at page 94]{MR0645390} and \cite[Theorem 1.3.1]{MR0184242}, the simplicial complex $X$ is  canonically a subcomplex  of the join 
\[
f^{-1}(v_0)*f^{-1}(v_1)*\dots*f^{-1}(v_k),
\]
where $v_0, v_1, \dots, v_k$ are the vertices of $\Delta$. If $X$ were a full simplex, then $f^{-1}(y)$ would be canonically homeomorphic to 
\begin{equation}\label{prod}
f^{-1}(v_0)\times \dots \times f^{-1}(v_k).
\end{equation}
 So we have such a product structure for each simplex $L$ contained in $f^{-1}(y)$. We want to apply this to the simplicial map $q\colon \mathcal{S}(M)\to [\mathcal{S}](M)$. To determine the preimage $q^{-1}(y)$, we first consider the preimage $q^{-1}(y)\cap \Delta$ in each simplex $\Delta\subset \mathcal{S}(M)$ that maps onto $\Delta([{\bf S}])$ and then they are glued together along the faces to give $q^{-1}(y)$. Let $q|_{\Delta}$ be the restriction of $q$ to the simplex $\Delta$ and let $\Delta_{[S_i]}$ be the face of $\Delta$ that is the preimage $q|_{\Delta}^{-1}([S_i])$. As in the general case (\ref{prod}), the preimage $q^{-1}(y)\cap \Delta$ is canonically homeomorphic to the product of simplices
\[
\Delta_{[S_0]}\times \Delta_{[S_1]}\times\cdots\times \Delta_{[S_k]}.
\]
The vertices of the simplex $\Delta_{[S_i]}\subset \mathcal{S}(M)$ are in the isotopy class of $S_i$ so its simplices are parallel spheres isotopic to the sphere $S_i$ and there is a natural {\it inside to outside} order (see \Cref{order}) on the vertices of each simplex in $\Delta_{[S_i]}$. 

Therefore, we will model $(\pi\circ q)^{-1}(x)$ for $x\in [\dot{\Delta}([{\bf S}])]$ by a multi-semi-simplicial set $ \mathcal{S}_{\bullet,\dots, \bullet}(M, [{\bf S}])$ whose realization is homeomorphic to $(\pi\circ q)^{-1}(x)$. The set $ \mathcal{S}_{n_0,\dots, n_k}(M, [{\bf S}])$ is the subset of $(k+\sum_i n_i)$-simplices of $\mathcal{S}(M)$ where exactly $(n_i+1)$ of its vertices lie over the $i$-th vertex of $x$ so there is natural order on vertices above the $i$-th vertex of $x$ for each $i$. On the other hand, the preimage of $\eta^{-1}(x)$, similar to \Cref{preimage}, is homeomorphic to 
 \[
(\pi\circ q)^{-1}(x)\hcoker \dD(M, \text{rel }\partial),
 \]
 which is in turn weakly equivalent to $||\mathcal{S}_{\bullet, \dots, \bullet}(M, [{\bf S}])\hcoker \dD(M, \text{rel }\partial)||$.

 Similar to \Cref{top}, we have a multi-semi-simplicial space $ \mathcal{S}^{\tau}_{\bullet,\dots, \bullet}(M, [{\bf S}])$ and an acyclic map
\[
k\colon ||\mathcal{S}_{\bullet, \dots, \bullet}(M, [{\bf S}])\hcoker \dD(M, \text{rel }\partial)||\to ||\mathcal{S}^{\tau}_{\bullet, \dots, \bullet}(M, [{\bf S}])\hcoker \tH(M, \text{rel }\partial)||.
\]
Since the map $||\mathcal{S}_{\bullet, \dots, \bullet}(M, [{\bf S}])\hcoker \dD(M, \text{rel }\partial)||\to \mathrm{B}\tH(M, \text{rel }\partial)$ factors through the map  $k$, to prove that $\eta^{-1}(x)$ is weakly homologically finite, it is enough to show that $||\mathcal{S}^{\tau}_{\bullet, \dots, \bullet}(M, [{\bf S}])\hcoker \tH(M, \text{rel }\partial)||$ is strongly homologically finite by putting it in a homotopy fiber sequence whose base and fiber are strongly homologically finite by induction on the number of prime factors. 
 
 By doing the bar construction model in each simplicial direction, we have homotopy fiber sequences similar to the homotopy fiber sequence \ref{fib4}. By applying \Cref{Kupers} we obtain a homotopy fiber sequence similar to (\ref{fib1}) whose fiber and the base are realizations of multi-semi-simplicial spaces with fewer simplicial directions to which we can apply induction on the number of simplicial directions. 
 
Let $M_1(S_0)$ and $M_2(S_0)$ be the submanifolds obtained by cutting $M$ along $S_0$ and $M_1(S_0)$ containing the boundary component $\partial_* M$. Suppose that $k'$ of spheres in $\{S_1, \dots, S_k\}$ lie in $M_1(S_0)$ and the rest are in $M_2(S_0)$. By doing the bar construction model and applying \Cref{Kupers}, we obtain a homotopy fiber sequence
 \[
 ||\mathrm{R}_{\bullet,\dots, \bullet}(M_2(S_0))||\to ||\mathcal{S}^{\tau}_{\bullet, \dots, \bullet}(M, [{\bf S}])\hcoker \tH(M, \text{rel }\partial)||\to ||\mathrm{L}_{\bullet,\dots, \bullet}(M_1(S_0))||,
 \]
 where the number of simplicial directions in $\mathrm{R}_{\bullet,\dots, \bullet}(M_2(S_0))$ and $\mathrm{L}_{\bullet,\dots, \bullet}(M_1(S_0))$ are respectively $k-k'$ and $k'$. Hence, it is easy to see that we can exhaust simplicial directions by considering homotopy fiber sequences and using \Cref{Kupers}. 
 
 To illustrate the idea, suppose $k=1$ and suppose the sphere $S_1$ lies in $M_2(S_0)$. Since the base point of $M$ lies in the component $M_1(S_0)$, we choose a base point for $M_2(S_0)$ that lies also on the boundary of $M$. Suppose $S_1$ separates $M_2(S_0)$ into two components $M_{2,1}(S_0,S_1)$ and $M_{2,2}(S_0,S_1)$ where the former contains the base point. Similar to \Cref{module1}, we have models for $\BD(M_{2,1}(S_0,S_1), \text{rel }\partial)$ and $\BD(M_{2,2}(S_0,S_1), \text{rel }\partial)$ as left and right $\mathfrak{D}$-module and we have a model for $\BD(M_{1}(S_0), \text{rel }\partial)$ as left $\mathfrak{D}$-module. Let us denote them respectively by $\mathcal{L}(M_{2,1}(S_0,S_1))$, $\mathcal{R}(M_{2,2}(S_0,S_1))$ and $\mathcal{L}(M_{1}(S_0))$. Then, similar to the weak equivalence (\ref{bar}), the iterated bar construction

 \[
||B_{\bullet}(\mathcal{L}(M_{1}(S_0)), \mathfrak{D}, B_{\bullet}(\mathcal{L}(M_{2,1}(S_0,S_1)), \mathfrak{D}, \mathcal{R}(M_{2,2}(S_0,S_1))))|| 
 \]
is weakly equivalent to $||\mathcal{S}^{\tau}_{\bullet, \bullet}(M, [{\bf S}])\hcoker \tH(M, \text{rel }\partial)||$. Hence, we have a fiber sequence
 \begin{equation}
\begin{gathered}
\begin{tikzpicture}[node distance=7cm, auto]
  \node (A) {$||B_{\bullet}(\mathcal{L}(M_{2,1}(S_0,S_1)), \mathfrak{D}, \mathcal{R}(M_{2,2}(S_0,S_1)))||$};
  \node (B) [right of=A] {$ ||\mathcal{S}^{\tau}_{\bullet,  \bullet}(M, [{\bf S}])\hcoker \tH(M, \text{rel }\partial)||$};
  \node (C) [below of=B, node distance=1.2cm] {$||B_{\bullet}(\mathcal{L}(M_{1}(S_0)), \mathfrak{D}, *)||.$};  

  \draw[->] (A) to node {$$} (B);
    \draw[->] (B) to node {$$} (C);

\end{tikzpicture}
\end{gathered}
\end{equation}
 But the base which is a one-sided bar construction that is weakly equivalent to $\BD(M_{1}(S_0), S_0, \text{rel }\partial M)$  by \Cref{Kupers}. 
 
 Hence, in general, we can use \Cref{Kupers} to reduce the number of simplicial directions by replacing one-sided bar constructions with the appropriate classifying space that ``makes the corresponding sphere boundary free''. By iterating this process for the base and the fiber and using \Cref{Kont} inductively, we conclude that the total space also is strongly homologically finite.

 \bibliographystyle{alpha}
\bibliography{reference}

\begin{thebibliography}{HKMR12}

\bibitem[BBS24]{boyd2024moduli}
Rachael Boyd, Corey Bregman, and Jan Steinebrunner.
\newblock Moduli spaces of 3-manifolds with boundary are finite.
\newblock {\em arXiv preprint arXiv:2404.12748}, 2024.

\bibitem[BK19]{bamler2019ricci}
Richard~H Bamler and Bruce Kleiner.
\newblock Ricci flow and contractibility of spaces of metrics.
\newblock {\em arXiv preprint arXiv:1909.08710}, 2019.

\bibitem[BK21]{bamler2019diffeomorphisms}
Richard~H Bamler and Bruce Kleiner.
\newblock Diffeomorphism groups of prime 3-manifolds.
\newblock {\em arXiv preprint arXiv:2108.03302}, 2021.

\bibitem[BL74]{MR0356103}
Dan Burghelea and Richard Lashof.
\newblock The homotopy type of the space of diffeomorphisms. {I}, {II}.
\newblock {\em Trans. Amer. Math. Soc.}, 196:1--36; ibid. 196\ (1974), 37--50,
  1974.

\bibitem[Bon83]{Bonahon}
Francis Bonahon.
\newblock Cobordism of automorphisms of surfaces.
\newblock {\em Ann. Sci. \'{E}cole Norm. Sup. (4)}, 16(2):237--270, 1983.

\bibitem[CdSR79]{MR513752}
Eug\'{e}nia C\'{e}sar~de S\'{a} and Colin Rourke.
\newblock The homotopy type of homeomorphisms of {$3$}-manifolds.
\newblock {\em Bull. Amer. Math. Soc. (N.S.)}, 1(1):251--254, 1979.

\bibitem[Cer61]{cerf1961topologie}
Jean Cerf.
\newblock Topologie de certains espaces de plongements.
\newblock {\em Bulletin de la Soci{\'e}t{\'e} Math{\'e}matique de France},
  89:227--380, 1961.

\bibitem[EK71]{edwards1971deformations}
Robert~D Edwards and Robion~C Kirby.
\newblock Deformations of spaces of imbeddings.
\newblock {\em Annals of Mathematics}, pages 63--88, 1971.

\bibitem[ERW19]{MR3995026}
Johannes Ebert and Oscar Randal-Williams.
\newblock Semisimplicial spaces.
\newblock {\em Algebr. Geom. Topol.}, 19(4):2099--2150, 2019.

\bibitem[FOT08]{felix2008algebraic}
Yves F{\'e}lix, John Oprea, and Daniel Tanr{\'e}.
\newblock {\em Algebraic models in geometry}.
\newblock OUP Oxford, 2008.

\bibitem[Ham74]{hamstrom1974homotopy}
Mary-Elizabeth Hamstrom.
\newblock Homotopy in homeomorphism spaces, {$ TOP $} and {$ PL$}.
\newblock {\em Bulletin of the American Mathematical Society}, 80(2):207--230,
  1974.

\bibitem[Hat76]{MR0420620}
Allen Hatcher.
\newblock Homeomorphisms of sufficiently large {$P^{2}$}-irreducible
  {$3$}-manifolds.
\newblock {\em Topology}, 15(4):343--347, 1976.

\bibitem[Hat81]{hatcher1981diffeomorphism}
A.~Hatcher.
\newblock On the diffeomorphism group of {$S^{1}\times S^{2}$}.
\newblock {\em Proc. Amer. Math. Soc.}, 83(2):427--430, 1981.

\bibitem[Hat83]{hatcher1983proof}
Allen~E Hatcher.
\newblock A proof of the {S}male conjecture, {{$\text{Diff}(S^3)\simeq O(4)$}}.
\newblock {\em Annals of Mathematics}, pages 553--607, 1983.

\bibitem[Hem76]{MR0415619}
John Hempel.
\newblock {\em {$3$}-{M}anifolds}.
\newblock Annals of Mathematics Studies, No. 86. Princeton University Press,
  Princeton, N. J.; University of Tokyo Press, Tokyo, 1976.

\bibitem[HKMR12]{MR2976322}
Sungbok Hong, John Kalliongis, Darryl McCullough, and J.~Hyam Rubinstein.
\newblock {\em Diffeomorphisms of elliptic 3-manifolds}, volume 2055 of {\em
  Lecture Notes in Mathematics}.
\newblock Springer, Heidelberg, 2012.

\bibitem[HL84]{MR780734}
Harrie Hendriks and Fran\c{c}ois Laudenbach.
\newblock Diff\'{e}omorphismes des sommes connexes en dimension trois.
\newblock {\em Topology}, 23(4):423--443, 1984.

\bibitem[HM87]{MR893801}
Harrie Hendriks and Darryl McCullough.
\newblock On the diffeomorphism group of a reducible {$3$}-manifold.
\newblock {\em Topology Appl.}, 26(1):25--31, 1987.

\bibitem[HM90]{hatcher1990finite}
Allen Hatcher and Darryl McCullough.
\newblock Finite presentation of 3-manifold mapping class groups.
\newblock In {\em Groups of Self-Equivalences and Related Topics}, pages
  48--57. Springer, 1990.

\bibitem[HM97]{MR1486644}
Allen Hatcher and Darryl McCullough.
\newblock Finiteness of classifying spaces of relative diffeomorphism groups of
  {$3$}-manifolds.
\newblock {\em Geom. Topol.}, 1:91--109, 1997.

\bibitem[Iva76]{MR0448370}
N.~V. Ivanov.
\newblock Groups of diffeomorphisms of {W}aldhausen manifolds.
\newblock {\em Zap. Nau\v cn. Sem. Leningrad. Otdel. Mat. Inst. Steklov.
  (LOMI)}, 66:172--176, 209, 1976.
\newblock Studies in topology, II.

\bibitem[Kir95]{kirby1995problems}
Rob Kirby.
\newblock Problems in low-dimensional topology.
\newblock In {\em Proceedings of Georgia Topology Conference, Part 2}.
  Citeseer, 1995.

\bibitem[KS77]{MR0645390}
Robion~C. Kirby and Laurence~C. Siebenmann.
\newblock {\em Foundational essays on topological manifolds, smoothings, and
  triangulations}.
\newblock Princeton University Press, Princeton, N.J.; University of Tokyo
  Press, Tokyo, 1977.
\newblock With notes by John Milnor and Michael Atiyah, Annals of Mathematics
  Studies, No. 88.

\bibitem[Kup15]{kupers2015proving}
Alexander Kupers.
\newblock Proving homological stability for homeomorphisms of manifolds.
\newblock {\em arXiv preprint arXiv:1510.02456}, 2015.

\bibitem[Kup19a]{kupers2019lectures}
Alexander Kupers.
\newblock Lectures on diffeomorphism groups of manifolds.
\newblock {\em University Lecture Notes}, page~59, 2019.

\bibitem[Kup19b]{kupers2019some}
Alexander Kupers.
\newblock Some finiteness results for groups of automorphisms of manifolds.
\newblock {\em Geometry \& Topology}, 23(5):2277--2333, 2019.

\bibitem[Las76]{lashof1976}
R.~Lashof.
\newblock Embedding spaces.
\newblock {\em Illinois J. Math.}, 20(1):144--154, 03 1976.

\bibitem[McD80]{mcduff1980homology}
Dusa McDuff.
\newblock The homology of some groups of diffeomorphisms.
\newblock {\em Commentarii Mathematici Helvetici}, 55(1):97--129, 1980.

\bibitem[Mil57]{MR0084138}
John Milnor.
\newblock The geometric realization of a semi-simplicial complex.
\newblock {\em Ann. of Math. (2)}, 65:357--362, 1957.

\bibitem[MN20]{mann2020dynamical}
Kathryn Mann and Sam Nariman.
\newblock Dynamical and cohomological obstructions to extending group actions.
\newblock {\em Mathematische Annalen}, 377(3):1313--1338, 2020.

\bibitem[Nar20]{nariman2020local}
Sam Nariman.
\newblock A local to global argument on low dimensional manifolds.
\newblock {\em Transactions of the American Mathematical Society},
  373(2):1307--1342, 2020.

\bibitem[Neu96]{neumann1996notes}
Walter~D Neumann.
\newblock {\em Notes on geometry and 3-manifolds}.
\newblock Topology Atlas, 1996.

\bibitem[Pal60]{MR0123338}
Richard~S. Palais.
\newblock Local triviality of the restriction map for embeddings.
\newblock {\em Comment. Math. Helv.}, 34:305--312, 1960.

\bibitem[Sma59]{MR0112149}
Stephen Smale.
\newblock Diffeomorphisms of the {$2$}-sphere.
\newblock {\em Proc. Amer. Math. Soc.}, 10:621--626, 1959.

\bibitem[Thu74]{thurston1974foliations}
William Thurston.
\newblock Foliations and groups of diffeomorphisms.
\newblock {\em Bulletin of the American Mathematical Society}, 80(2):304--307,
  1974.

\bibitem[Wal68]{MR224099}
Friedhelm Waldhausen.
\newblock On irreducible {$3$}-manifolds which are sufficiently large.
\newblock {\em Ann. of Math. (2)}, 87:56--88, 1968.

\bibitem[Whi12]{whitehead2012elements}
George~W Whitehead.
\newblock {\em Elements of homotopy theory}, volume~61.
\newblock Springer Science \& Business Media, 2012.

\bibitem[Wil66]{MR0184242}
Robert~E. Williamson, Jr.
\newblock Cobordism of combinatorial manifolds.
\newblock {\em Ann. of Math. (2)}, 83:1--33, 1966.

\end{thebibliography}
\end{document}